\documentclass[reqno,draft]{amsart}
\usepackage{amsmath, amsthm, amssymb, dsfont}
\usepackage[a4paper]{geometry}					
\usepackage{graphicx}
\usepackage{tikz}
\usepackage{capt-of}
\usepackage{enumerate}

\usepackage{latexsym}
\usepackage{bbm}
\usepackage{mathtools}
\usepackage{array}	

\usepackage{comma}
\usepackage{cutwin}
\usepackage{pstricks}
\usepackage{pst-plot}
\usepackage{pgfplots}
\usepackage{float}
\usepackage[final,hidelinks]{hyperref}

\newcommand{\commentblock}[1]{}

\numberwithin{equation}{section}
\newcounter{assumptions}

\usepackage{changes}

\theoremstyle{plain}
\newtheorem{theorem}{Theorem}[section]
\newtheorem{lemma}[theorem]{Lemma}
\newtheorem{corollary}[theorem]{Corollary}
\newtheorem{proposition}[theorem]{Proposition}
\theoremstyle{remark}
\newtheorem{remark}[theorem]{Remark}
\theoremstyle{definition}

\newtheorem{example}[theorem]{Example}

\newcommand{\N}{\mathbb{N}}
\newcommand{\Z}{\mathbb{Z}}

\newcommand{\R}{\mathbb{R}}

\newcommand{\Rnn}{\mathbb{R}_\geq}
\newcommand{\C}{\mathbb{C}}

\DeclareMathOperator{\Real}{\mathrm{Re}}
\DeclareMathOperator{\Imag}{\mathrm{Im}}
\newcommand{\imag}{\mathrm{i}}

\newcommand{\A}{\mathcal{A}}

\newcommand{\cG}{\mathcal{G}}

\newcommand{\cH}{\mathcal{H}}
\newcommand{\cL}{\mathcal{L}}

\newcommand{\Prob}{\mathbf{P}}

\newcommand{\E}{\mathbf{E}}
\newcommand{\Psf}{\mathsf{P}}

\newcommand{\eqdist}{%
  \mathrel{\vbox{\offinterlineskip\ialign{%
    \hfil##\hfil\cr
    $\scriptscriptstyle\mathrm{d}$\cr
    \noalign{\kern.1ex}
    $=$\cr
}}}}

\newcommand{\1}{\mathbbm{1}}
\newcommand{\comp}{\mathsf{c}}

\newcommand{\interior}[1]{%
	{\kern0pt#1}^{\mathrm{o}}%
}

\newcommand{\distto}{%
  \mathrel{\vbox{\offinterlineskip\ialign{%
    \hfil##\hfil\cr
    $\scriptscriptstyle\mathrm{d}$\cr
    \noalign{\kern-.05ex}
    $\to$\cr
}}}}
\newcommand{\Probto}{%
  \mathrel{\vbox{\offinterlineskip\ialign{%
    \hfil##\hfil\cr
    $\scriptscriptstyle\Prob$\cr
    \noalign{\kern-.05ex}
    $\to$\cr
}}}}
\newcommand{\TVto}{%
  \mathrel{\vbox{\offinterlineskip\ialign{%
    \hfil##\hfil\cr
    $\scriptscriptstyle\mathrm{TV}$\cr
    \noalign{\kern-.05ex}
    $\to$\cr
}}}}

\newcommand{\dx}{\mathrm{d} \mathit{x}}
\newcommand{\dy}{\mathrm{d} \mathit{y}}
\newcommand{\dz}{\mathrm{d} \mathit{z}}

\newcommand{\transp}{\mathsf{T}}
\newcommand{\adj}{\mathrm{adj}}
\newcommand{\Id}{\mathrm{I}_d}

\newcommand{\e}{\mathsf{e}}

\newcommand{\bphi}{\boldsymbol{\varphi}}

\newcommand{\bxi}{\boldsymbol{\xi}}
\newcommand{\bzeta}{\boldsymbol{\zeta}}
\newcommand{\bU}{\boldsymbol U}

\newcommand{\bmu}{\boldsymbol{\mu}}
\newcommand{\bnu}{\boldsymbol{\nu}}

\newcommand{\Ip}{\mathbf{I}_{p}}

\newcommand{\defeq}{\vcentcolon=}
\newcommand{\eqdef}{=\vcentcolon}

\newcommand{\Exp}{\mathrm{Exp}}

\newcommand{\familytree}{\mathcal{T}}
\newcommand{\I}{\mathcal{I}}

\renewcommand{\k}{\mathtt{k}}
\newcommand{\p}{[ \, p \,]}
\newcommand{\Dom}{\mathcal{D}}
\newcommand{\cZ}{\mathcal{Z}}

\begin{document}
\title[Asymptotic expansion of solutions to Markov renewal equations]{Asymptotic expansions of solutions to Markov renewal equations
and their application to general branching processes}

\author{Konrad Kolesko}
\address{}
\email{konrad.kolesko@math.uni.wroc.pl}

\author{Matthias Meiners}
\email{matthias.meiners@math.uni-giessen.de}
\thanks{Corresponding author: Matthias Meiners, \texttt{matthias.meiners@math.uni-giessen.de}.}

\author{Ivana Tomic}
\email{ivana.tomic-2@math.uni-giessen.de}

\begin{abstract}
We consider the Markov renewal equation $F(t) = f(t) + \bmu*F(t)$
for vector-valued functions $f,F: \R \to \R^{p}$ and a $p \times p$ matrix
$\bmu$ of locally finite measures $\mu^{i,j}$ on $[0,\infty)$, $i,j=1,\ldots,p$.
Sgibnev [Semimultiplicative estimates for the solution of the multidimensional renewal equation.
{\em Izv.\ Ross.\ Akad.\ Nauk Ser.\ Mat.}, 66(3):159--174, 2002]
derived an asymptotic expansion for the solution $F$ to the above equation.
We give a new, more elementary proof of Sgibnev's result, which also covers the reducible case.
As a corollary, we infer an asymptotic expansion for the mean of a multi-type general branching process
with finite type space counted with random characteristic.

Finally, some examples are discussed that illustrate phenomena of multi-type branching.
\smallskip

\\{\bf Subclass:} MSC: Primary: 45M05. Secondary: 60K15, 60J80
\end{abstract}

\maketitle

\section{Introduction}

In the paper at hand, we study the $p$-dimensional \emph{Markov renewal equation}
\begin{equation}	\label{eq:Markov renewal}
F(t) = f(t) + \bmu*F(t)
\end{equation}
for some $p \in \N = \{1,2,\ldots\}$, (column) vector-valued functions $f,F: \R \to \R^{p}$ and a $p \times p$ matrix
\begin{align*}
\bmu
=\begin{pmatrix}
\mu^{1,1}& \mu^{1,2}   & \cdots & \mu^{1,p}  \\
\mu^{2,1} & \mu^{2,2} & \cdots & \mu^{2,p} \\
\vdots & \ddots     & \ddots & \vdots  \\
\mu^{p,1}  & \mu^{p,2} &  \cdots  & \mu^{p,p}
\end{pmatrix}
\end{align*}
of locally finite measures $\mu^{i,j}$ on $[0,\infty)$, $i,j \in \p \defeq \{1,\ldots,p\}$.
The convolution of the matrix of measures $\bmu$ and the vector-valued function $F$
is the column vector whose $i$-th entry is given by
\begin{equation*}
(\bmu*F)_i(t) = \sum_{j=1}^p \mu^{i,j}*F_j(t) = \sum_{j=1}^p \int F_j(t-x) \, \mu^{i,j}(\dx),	\quad	i=1,\ldots,p.
\end{equation*}
In \eqref{eq:Markov renewal}, the function $f$ is given while $F$ is considered unknown.
Under mild assumptions on $f$ and $\bmu$, see Proposition \ref{Prop:existence and uniqueness MRE} below,
there is a unique solution $F$ to \eqref{eq:Markov renewal}.
The Markov renewal equation occurs in numerous applications,
for example in the context of multi-type branching processes,
see Section \ref{sec:applications} below for further details.
The classical Markov renewal theorem \cite{Crump:1970a,Crump:1970b,Sevastjanov+Chistyakov:1971}
gives the leading order of the solution $F$ as $t \to \infty$.
In the realm of branching processes, this corresponds to the leading order of the mean of the process.
There are further extensions to a general type space instead of a finite one, see e.g.\ \cite{Athreya+McDonald+Ney:1978,Athreya+Ney:1978,Nummelin:1978}.

Under suitable assumptions, discussed in detail in Section \ref{subsec:literature},
Sgibnev \cite[Theorem 3]{Sgibnev:2002} proved an expansion for the solution $F$ of the form
\begin{align}	\label{eq:Sgibnev's expansion} 
F(t) = \sum_{\lambda \in \Lambda} \sum_{k=0}^{\k(\lambda)-1} C_{\lambda,k} t^k e^{\lambda t} + \Delta_f(t)
\end{align}
where $\Lambda \subseteq \C$ is a certain finite set of solutions to the characteristic equation \eqref{eq:characteristic equation} below,
$\k(\lambda) \in \N$,
$C_{\lambda,k}$ is a suitable real $p \times p$ matrix, and $\Delta_f(t)$ is a lower-order remainder term.
His proof is based on a representation theorem for homomorphisms on a specific Banach algebra.
The proof is extremely elegant but may go beyond the usual prior knowledge of readers from the probability community.
The aim of this paper is to give a more elementary proof of Sigbnev's result,
which is essentially based only on basic results about Laplace transforms and from complex analysis.

\section{Main results}

Before the main results can be formulated, some preparations are required.

\subsection{Assumptions and preliminaries}
In this section we formulate the assumptions that are made in some of the results or even throughout this text.

\subsubsection*{Laplace transforms}
Throughout the paper, we write $\cL$ for the Laplace operator and apply it component-wise
to vectors and matrices of measures on $[0,\infty)$, e.g., $\cL\bmu(\theta) = (\cL\mu^{i,j}(\theta))_{i,j\in\p}$
where
\begin{align*}
\cL\mu^{i,j}(\theta) \defeq \int e^{-\theta x} \, \mu^{i,j}(\dx),	\quad	\theta \geq 0.
\end{align*}
We write $\mu^{i,j}(t)$ for $\mu^{i,j}([0,t])$ and $\bmu(t)$ for $\bmu([0,t])=(\mu^{i,j}([0,t]))_{i,j \in \p}$, $t \geq 0$.
Further, $\mu^{i,j}(\infty) \defeq \lim_{t \to \infty} \mu^{i,j}(t) = \mu^{i,j}([0,\infty))$
and $\bmu(\infty) = (\mu^{i,j}(\infty))_{i,j \in \p}$.
With this notation
\begin{align*}
\cL\bmu(0) = (\cL\mu^{i,j}(0))_{i,j\in\p} = \bmu(\infty)
\end{align*}
is the matrix of total masses, which may well have infinite entries.

Whenever $\cL\mu^{i,j}(\theta)<\infty$ for some $\theta \in \R$, then
$\cL\mu^{i,j}(z) = \int e^{-z x} \, \mu^{i,j}(\dx)$ converges absolutely
for every $z \in \C$ with $\Real(z) \geq \theta$.
We write
\begin{align*}
\Dom(\cL\bmu) \defeq \{z \in \C: {\textstyle \int} e^{-\Real(z) x} \, \mu^{i,j}(\dx) < \infty \text{ for all } i,j \in \p\}
\end{align*}
for the domain of $\cL\bmu$. Then, in general, $\Dom(\cL\bmu)$ is either empty, a
half-space, or $\C$.
Throughout the paper, we make the basic assumption that the domain of finiteness
of the Laplace transform $\cL\bmu$ is non-empty:

\begin{enumerate}[{\bf{(A}1)}]
	\setcounter{enumi}{\value{assumptions}}
	\item
	There exists a $\vartheta \in \interior{\Dom(\cL\bmu)}$. In particular,
	\begin{align}	
	\cL\mu^{i,j}(\vartheta) = \int e^{-\vartheta x} \, \mu^{i,j}(\dx) < \infty
	\end{align}
	for all $i,j \in \p$.
	\label{ass:first moment}
\setcounter{assumptions}{\value{enumi}}
\end{enumerate}

\subsubsection*{Primitivity}

Recall that a nonnegative $p \times p$ matrix $A = (a_{ij})_{i,j\in\p}$ is called \emph{primitive}
if there exists a nonnegative integer $n$ such that the matrix power $A^n$ has only positive entries.
Notice that if $\widetilde A$ is the incidence matrix of $A$, that is,
the matrix with all positive (including the infinite) entries replaced by ones, then $A$ is primitive iff its incidence matrix is.

All matrices $\cL\bmu(\theta)$, $\theta \in \Dom(\cL\bmu) \cap \R$
have the same incidence matrix,
which may also be seen as the incidence matrix of $\bmu$,
and we therefore denote it by $\widetilde \bmu$.
We call equation \eqref{eq:Markov renewal}
\emph{primitive} if the matrix $\widetilde \bmu$ is primitive.
The paper at hand is not limited to the primitive case,
but since this case is of particular interest, we shall give it special attention occasionally.

\subsubsection*{Spectral radius}

For a real or complex square matrix $A$, we write
\begin{align}	\label{eq:rho_A}
\rho_A \defeq \rho(A) \defeq \sup\{|\lambda|: \, \lambda \in \C \text{ is an eigenvalue of } A\}
\end{align}
for its spectral radius. Gelfand's formula gives the alternative representation
\begin{align}	\label{eq:Gelfand}
\rho(A) = \lim_{n \to \infty} \|A^n\|^{\frac1n} = \inf_{n \in \N} \|A^n\|^{\frac1n}.
\end{align}
If $A$ is nonnegative and primitive,
then $\rho_A$ is the Perron-Frobenius eigenvalue of $A$, see \cite[Theorem 1.1]{Seneta:1981}.
In the reducible case, $\rho_A$ is still an eigenvalue of $A$, see \cite[Section 15.5, Theorem 1]{Lancaster+Tismenetsky:1985}.

Throughout the paper, we make the following assumption.

\begin{enumerate}[{\bf{(A}1)}]
	\setcounter{enumi}{\value{assumptions}}
	\item
	The spectral radius $\rho_{\bmu(0)}$ of the matrix $\bmu(0) = (\mu^{i,j}(0))_{i,j\in\p}$
	satisfies $\rho_{\bmu(0)} < 1$.				\label{ass:subcritical instant offspring}
\setcounter{assumptions}{\value{enumi}}
\end{enumerate}

\subsubsection*{Existence and uniqueness of solutions}

The Markov renewal measure defined by
\begin{equation}	\label{eq:U}
\bU = (U^{i,j})_{i,j\in\p} \defeq \sum_{n=0}^\infty \bmu^{*n}
\end{equation}
is of fundamental importance when analyzing Markov renewal equations.
The following basic existence and uniqueness result holds.

\begin{proposition}	\label{Prop:existence and uniqueness MRE}
Let $\bmu=(\mu^{i,j})_{i,j\in\p}$ be a $p \times p$ matrix
of locally finite measures on $[0,\infty)$ satisfying (A\ref{ass:subcritical instant offspring})
with associated Markov renewal measure $\bU$.
Then the following assertions hold.
\begin{enumerate}[(a)]
	\item
		The Markov renewal function $\bU(t)$ is finite at every $t \geq 0$.
	\item
		If, additionally, (A\ref{ass:first moment}) holds,
		then, for any sufficiently large $\theta \in \Dom(\cL\bmu)$, we have $\rho(\cL\bmu(\theta))<1$ and,
		for any such $\theta$,
		it holds that $U^{i,j}(t) = o(e^{\theta t})$ as $t \to \infty$ for all $i,j\in\p$.
	\item
		Let $f = (f_1,\ldots,f_p)^\transp$ be such that each $f_{i}$ is a locally bounded, Borel-measurable function supported on $[0,\infty)$,
		$i=1,\ldots,p$.
		Then there is a unique locally finite solution $F:\R \to \R^{p}$
		to the Markov renewal equation \eqref{eq:Markov renewal}, namely,
		\begin{equation}	\label{eq:solution MRE explicit}
		F(t) = \bU*f(t),	\quad	t \in \R.
		\end{equation}
\end{enumerate}
\end{proposition}

The result is classical in the case univariate case $p=1$ and covered by textbooks,
we refer, e.g.\ to \cite[Theorem 3.5.1]{Resnick:1992}.
The multivariate case is covered, e.g., in \cite{Crump:1970a} in the irreducible case.
The reducible case is somewhat implicitly contained in \cite{Crump:1970b}.
For the reader's convenience, we sketch the proof in Section \ref{subsec:existence and uniqueness} below.

The goal of the paper is to derive an asymptotic expansion of the solution $F$ to \eqref{eq:Markov renewal}
under appropriate assumptions on $f$ and $\bmu$.
One of the prime examples for solving a Markov renewal equation is,
of course, the renewal measure $\bU$ itself.
Indeed,
\begin{equation}	\label{eq:renewal for U}
\bU(t) = \sum_{n=0}^\infty \bmu^{*n}(t) = \bmu^{*0}(t) + \sum_{n=1}^\infty \bmu^{*n}(t) = \Ip \1_{[0,\infty)}(t) + \bmu*\bU(t)
\end{equation}
for any $t \geq 0$ where $\Ip$ denotes the $p \times p$ identity matrix.
The starting point for the derivation of an asymptotic expansion for $F$ will be such an expansion for $\bU$.

\subsubsection*{The lattice type}
Throughout this work, we consistently distinguish between two cases:
\begin{itemize}
	\item
		The \emph{lattice case}, in which, for some $h>0$,
		the measures $\mu^{i,j}$, $i,j \in \p$ are concentrated on $h\Z$.
	\item
		The \emph{non-lattice case}, in which for all $h>0$ and $h_1,...,h_p\in[0,h)$,
		$\mu^{i,j}((h_j-h_i+h\Z)^\comp)>0$ is satisfied for some $i,j\in\p$.
\end{itemize}

In the lattice case, we assume without loss of generality that $h=1$ is maximal
with the property that all the $\mu^{i,j}$, $i,j \in \p$ are concentrated on $h\Z$.
The general lattice case can be reduced to this one by appropriate scaling.

\subsubsection*{Roots of the characteristic equation $\det(\Ip - \cL\bmu(\lambda)) = 0$}

The asymptotic expansion of the solution $F(t)$ to \eqref{eq:Markov renewal} as $t \to \infty$
is closely related to the following
\emph{characteristic equation}:
\begin{align}	\label{eq:characteristic equation}
\det(\Ip - \cL\bmu(\lambda)) = 0,
\end{align}
where $\Ip$ denotes the $p \times p$ identity matrix and $\lambda \in \Dom(\cL\bmu)$.
Define
\begin{align*}
\Lambda \defeq \{\lambda \in \Dom(\cL\bmu): \det(\Ip - \cL\bmu(\lambda)) = 0\}
\end{align*}
in the non-lattice case, and $\Lambda \defeq \{\lambda = \theta + \imag \eta \in \Dom(\cL\bmu): -\pi < \eta \leq \pi \text{ and } \det(\Ip - \cL\bmu(\lambda)) = 0\}$ in the lattice case.
Notice that $\Lambda$ is a closed subset of $\Dom(\cL\bmu)$ since the determinant is a continuous functional
on $\C^{p \times p}$ equipped with the matrix norm.

Whenever $\det(\Ip-\cL\bmu(z)) \not = 0$,
we can invert $\Ip-\cL\bmu(z)$ using the formula
\begin{align}	\label{eq:inverse via adjoint}
(\Ip-\cL\bmu (z))^{-1}=\frac{1}{\det(\Ip-\cL\bmu(z))}\; \adj(\Ip-\cL\bmu(z)),
\end{align}
where $\adj(A)$ denotes the adjoint of the matrix $A$.
Therefore, each entry of $(\Ip-\cL\bmu(z))^{-1}$ has a pole with multiplicity at most
the multiplicity of the zero of $\det(\Ip-\cL\bmu(z))$ at $z=\lambda$.
Write $\k(\lambda)\in\N$ for the maximal multiplicity of the poles
of the entries of $(\Ip-\cL\bmu (z))^{-1}$ at the point $\lambda\in\Lambda$,
that is, $\k(\lambda)$ is the smallest natural number
such that all entries of $(z-\lambda)^{\k(\lambda)}(\Ip-\cL\bmu (z))^{-1}$ are holomorphic in some neighborhood of $\lambda$.
For completeness, we set $\k(z) = 0$ for $z \not \in \Lambda$.

We call a root $\alpha \in \Lambda \cap \R$ with $\rho_{\cL\bmu(\alpha)}=1$
\emph{Malthusian parameter} as it represents the maximum exponential growth rate of the expected population size
when applying the main result to branching processes. 
Provided that (A\ref{ass:first moment}) and (A\ref{ass:subcritical instant offspring}) are satisfied,
the existence of a Malthusian parameter will be guaranteed by the following assumption:
\begin{enumerate}[{\bf{(A}1)}]
	\setcounter{enumi}{\value{assumptions}}
	\item
	There exists $\theta \in \Dom(\cL\bmu) \cap \R$ with $1 \leq \rho_{\cL\bmu(\theta)} < \infty$.
	\label{ass:Lmu(vartheta) geq 1}
\setcounter{assumptions}{\value{enumi}}
\end{enumerate}

The following proposition establishes some fundamental properties of the set $\Lambda$:

\begin{proposition}	\label{Prop:largest root}
Suppose that (A\ref{ass:first moment}) and (A\ref{ass:subcritical instant offspring}) hold.
\begin{enumerate}[(a)]
	\item
		There exists at most one $\alpha \in \R$ with $\rho_{\cL\bmu(\alpha)}=1$.
	\item
		If (A\ref{ass:Lmu(vartheta) geq 1}) holds, then there exists a Malthusian parameter $\alpha$
		and $\Real(\lambda) \leq \alpha$ for every $\lambda \in \Lambda$.
	\item
		Condition (A\ref{ass:Lmu(vartheta) geq 1}) is necessary for $\Lambda$ to be non-empty.
	\item
		In the primitive case, if $\alpha \in \interior{\Dom(\cL\bmu)}$,
		then $\alpha$ is a pole of order one of $(\Ip-\cL\bmu(z))^{-1}$.
		Further, $\Real(\lambda) < \alpha$ for every $\lambda \in \Lambda \setminus \{\alpha\}$.
\end{enumerate}
\end{proposition}

In Proposition \ref{Prop:largest root}, we say that $\alpha$ is a pole of order one of $(\Ip-\cL\bmu(z))^{-1}$.
By this we mean that all entries of $(\Ip-\cL\bmu(z))^{-1}$ are either holomorphic at $z=\alpha$
or have a pole of order at most one while there is at least one entry of $(\Ip-\cL\bmu(z))^{-1}$
that has a pole of order one at $\alpha$.
More generally, if $A(z)$ is a matrix of meromorphic functions, we say that $A$ has a pole
of order $n \in \N$ at $\lambda$ if all entries of $A$ are either holomorphic at $\lambda$
or have a pole of order at most $n$ while there is at least one entry that has a pole of order exactly $n$. 
The proof of the proposition is postponed to Section \ref{subsec:varrho}.

Finally, for $\theta \in \R$, we define
\begin{align*}
\Lambda_\theta \defeq \{\lambda \in \Lambda: \Real(\lambda) > \theta\}.
\end{align*}

\subsection{Asymptotic expansions}

\subsubsection*{The non-lattice case}

We first consider the non-lattice case. More precisely, we work using the assumption
\begin{align}	\label{eq:sup||(Ip-Lmu(theta+ieta)^-1||<infty}
\sup_{\eta\in \R} \|(\Ip-\cL\bmu(\vartheta+\imag\eta))^{-1}\| < \infty
\end{align}
and some $\vartheta \in \interior{\Dom(\cL\bmu)} \cap \R$.
Notice that \eqref{eq:sup||(Ip-Lmu(theta+ieta)^-1||<infty} implies that there is no $\lambda \in \Lambda$ with $\Real(\lambda)=\vartheta$.
We assume without loss of generality that $\vartheta$ in \eqref{eq:sup||(Ip-Lmu(theta+ieta)^-1||<infty}
and (A\ref{ass:first moment}) coincide.

Notice that \eqref{eq:sup||(Ip-Lmu(theta+ieta)^-1||<infty} implies $\Real(\lambda) \not = \vartheta$ for every $\lambda \in \Lambda$.
The condition may seem technical at first glance.
However, we discuss it following our main result, Theorem \ref{Thm:Markov renewal theorem F(t) non-lattice} below,
and relate it in particular to the assumptions of Sgibnev \cite{Sgibnev:2002}.
The following expansion for the renewal function $t \mapsto \bU(t) = \sum_{n=0}^\infty \bmu^{*n}(t)$
is the starting point for all further investigations.

\begin{theorem}	\label{Thm:Markov renewal theorem U(t) non-lattice}
Let $\bmu=(\mu^{i,j})_{i,j\in\p}$ be a $p \times p$ matrix
of locally finite measures on $[0,\infty)$ satisfying (A\ref{ass:first moment}), (A\ref{ass:subcritical instant offspring}) and
\eqref{eq:sup||(Ip-Lmu(theta+ieta)^-1||<infty}.
Further, suppose that $\Lambda_\vartheta$ is finite.
Then there exist matrices $C_{\lambda,k} \in \R^{p \times p}$, $k=0,\ldots,\k(\lambda)-1$, $\lambda \in \Lambda_\vartheta$,
$C_{0,\k(0)}$ such that
\begin{equation}	\label{eq:asymptotic expansion U(t) non-lattice}
\bU(t)
= \sum_{\lambda\in\Lambda_{\vartheta}}e^{\lambda  t} \sum_{k=0}^{\k(\lambda)-1} \!\!\! t^k C_{\lambda, k}
+ t^{\k(0)} C_{0,\k(0)} + O(t e^{\vartheta t})
\quad	\text{as } t \to \infty,
\end{equation}
where the error bound $O(t e^{\vartheta t})$ as $t \to \infty$ applies entry by entry.
\end{theorem}

\begin{remark}	\label{Rem:C_lambda,k}
We can give formulae for the matrices $C_{\lambda,k}$.
Indeed, fix some $\lambda \in \Lambda$ with $\Real(\lambda) > 0$.
Recall that $\k(\lambda)\in\N$ is the maximal multiplicity of the poles
of the entries of $(\Ip-\cL\bmu (z))^{-1}$ at the point $\lambda\in\Lambda$, so we may write
\begin{equation}	\label{eq:expansion of I_p-Lmu inverse at lambda}
(\Ip-\cL\bmu (z))^{-1} = A_{\lambda,\k(\lambda)} (z-\lambda)^{-\k(\lambda)} + \ldots + A_{\lambda,1} (z-\lambda)^{-1} + H_\lambda(z)
\end{equation}
for all $z \not = \lambda$ in a small disc $D(\lambda)$ centered at $\lambda$. Here, $H_\lambda(z)$ is holomorphic in $D(\lambda)$
and $A_{\lambda,\k(\lambda)}, \ldots, A_{\lambda,1}$ are complex $p \times p$ matrices.
Further, $A_{\lambda,\k(\lambda)}$ has at least one non-zero entry. Then
\begin{equation*}
C_{\lambda,k}
= \frac1{k!} \sum_{n=0}^{\k(\lambda)-1-k} \frac{(-1)^{n}}{n!\lambda^{n+1}} A_{\lambda,n+k+1}.
\end{equation*}
\end{remark}

From Theorem \ref{Thm:Markov renewal theorem U(t) non-lattice}, using essentially integration by parts, we can derive an expansion
for solutions to general Markov renewal equations of the form \eqref{eq:Markov renewal}.
Before we can formulate this result, some notation needs to be introduced.
For a function $f:\R \mapsto \R$ we define the total variation function $\mathrm{V}\!f$ by
\begin{equation}	\label{eq:def of Vf}
\mathrm{V}\!f(x) \defeq \sup\bigg\{\sum_{j=1}^n |f(x_j)-f(x_{j-1})|: -\infty < x_0 < \ldots < x_n \leq x,\ n \in \N \bigg\}
\end{equation}
for $x \in \R$. If $f = (f_1,\ldots,f_p)^\transp$ is vector-valued, we write $\mathrm{V}\!f$ for the vector $(\mathrm{V}\!f_1,\ldots,\mathrm{V}\!f_p)^\transp$.
 
\begin{theorem} \label{Thm:Markov renewal theorem F(t) non-lattice}
Let $\bmu=(\mu^{i,j})_{i,j\in\p}$ be a $p \times p$ matrix
of locally finite measures on $[0,\infty)$ satisfying (A\ref{ass:first moment}), (A\ref{ass:subcritical instant offspring}) and
\eqref{eq:sup||(Ip-Lmu(theta+ieta)^-1||<infty}. Further, suppose that $\Lambda_\vartheta$ is finite.
Let $f:\R\to\R^{p}$ be a column vector-valued right-continuous function with existing left limits
vanishing on the negative halfline with the property
\begin{align} \label{eq:int e^-theta x f(x) < infty}
\int_0^\infty  e^{-\theta x} \, \mathrm{V}\! f (x)  \, \dx < \infty
\end{align}
for some $\theta > \vartheta$ with $\Lambda_\vartheta=\Lambda_\theta$.
Then there exist matrices $B_{\lambda,k} \in \R^{p \times p}$, for $k=0,\ldots,\k(\lambda)-1$, $\lambda \in \Lambda_\vartheta$,
and vectors $b_{\lambda,k,f} \in \R^p$, such that
\begin{align}
F(t) &=
\sum_{\lambda\in\Lambda_{\theta}} \!e^{\lambda t} \!\! \sum_{k=0}^{\k(\lambda)-1} \!\!\! B_{\lambda, k} \! \int \limits_{0}^t \! f(x) (t-x)^k e^{-\lambda x} \, \dx
+ O(e^{\theta t})	\label{eq:expansion F(t) non-lattice}\\	
&{=\sum_{\lambda\in\Lambda_{\theta}} \!e^{\lambda t} \!\! \sum_{k=0}^{\k(\lambda)-1} \!\!\! B_{\lambda, k} \! \int \limits_{0}^{\infty} \! f(x) (t-x)^k e^{-\lambda x} \, \dx
+ O(e^{(\theta+\varepsilon) t})}	\label{eq:expansion F(t) non-lattice 2}\\
&= \sum_{\lambda\in\Lambda_{\theta}} e^{\lambda t} \!\! \sum_{k=0}^{\k(\lambda)-1} \!\!\! t^k b_{\lambda, k,f}
+  O(e^{(\theta+\varepsilon) t})	\nonumber
\end{align}
as $t \to \infty$, where the error bounds $O(e^{\theta t})$ and $O(e^{(\theta+\varepsilon) t})$ as $t \to \infty$ are to be read component by component
and where $\varepsilon>0$ can be chosen arbitrarily small, in particular, $\theta+\varepsilon < \Real(\lambda)$ for every $\lambda \in \Lambda_\theta$.
\end{theorem}

\begin{remark}	\label{Rem:coefficients expansion for F(t) non-lattice}
The coefficients $B_{\lambda,k}$ and $b_{\lambda,k,f}$
in the expansions \eqref{eq:expansion F(t) non-lattice} and \eqref{eq:expansion F(t) non-lattice}
can be given explicitly in terms of moments of the function $f$ and the coefficients $C_{\lambda,k}$ from the asymptotic expansion of $\bU(t)$
given in Theorem \ref{Thm:Markov renewal theorem U(t) non-lattice} and Remark \ref{Rem:C_lambda,k}. 
In particular, if $\Real(\lambda)>0$,
\begin{align}
	\label{eq:B_lambda_k}
	B_{\lambda,k}=(k+1)C_{\lambda,k+1} + \lambda C_{\lambda,k}.
\end{align}
\end{remark}

\subsubsection*{The lattice case}
In the present subsection, we assume that the measures $\mu^{ij}$, $i,j \in \p$ forming the entries of the ($p\times p$)-matrix $\bmu$ are concentrated on $\N_0$.
In this case, we consider the discrete solution $F(\{n\})=\bU * f(\{n\})$ to the Markov renewal equation \eqref{eq:Markov renewal}.
To determine its asymptotic behavior, it is convenient to use generating functions rather than Laplace transforms.
We define the generating function of the $p\times p$ matrix $\bmu$ of measures as 
\begin{align}	\label{eq: generating function}
\cG\bmu(z) \defeq \sum_{n=0}^{\infty}\bmu(\{n\}) z^n
\end{align}
for all $z\in\C$ for which the series is absolutely convergent.
Notice that $\cG\bmu(e^{-z})=\cL\bmu(z)$ and $\cG\bmu(e^{-\vartheta})<\infty$ due to assumption (A\ref{ass:first moment}).
Consequently, the entries of the series \eqref{eq: generating function}
define holomorphic functions on an open disc that contains $\{|z| \leq e^{-\vartheta}\}$. 
There is the following relationship between the poles of $\cL\bmu$ and $\cG\bmu$.

\begin{lemma} \label{Lem: Multiplicities of lattice/non-lattice}
Suppose that (A\ref{ass:first moment}) and (A\ref{ass:subcritical instant offspring}) hold. If $z=\lambda$ is a pole of $(\Ip-\cL\bmu(z))^{-1}$ with multiplicity $\k(\lambda)$, then $z=e^{-\lambda}$ is pole of $(\Ip-\cG\bmu(z))^{-1}$ again with multiplicity $\k(\lambda)$.
\end{lemma}

For $i,j\in \p$, we denote by $b^{i,j}_{\lambda,k}$ the coefficient of $(z-e^{-\lambda})^{-k}$ in the Laurent series of the $(i,j)$-th entry of
$(\Ip-\cG\bmu(z))^{-1}$ around $z=e^{-\lambda}$
and define $B_{\lambda,k} \defeq (b^{i,j}_{\lambda,k})_{i,j\in\p}$, $k=1,\ldots,\k(\lambda)-1$.
Note that the coefficients $(B_{\lambda,k})_{k=1,\ldots,\k(\lambda)-1}$
depend only on $(A_{\lambda,k})_{k=1,\ldots,\k(\lambda)-1}$ and $\lambda$.
With this notation,
\begin{equation}
(\Ip-\cG\bmu(z))^{-1}=\sum_{k=1}^{\k(\lambda)}B_{\lambda,k} (z-e^{-\lambda})^{-k}+H_\lambda(z),
\end{equation}
where $H_\lambda(z)$ is holomorphic in a neighborhood of $e^{-\lambda}$.

\begin{theorem}	\label{Thm:Markov renewal theorem F(n) lattice}
Let $\bmu=(\mu^{i,j})_{i,j\in\p}$ be a $p \times p$ matrix
of locally finite measures concentrated on $\N_0$ such that (A\ref{ass:first moment}) and (A\ref{ass:subcritical instant offspring}) hold.
Further, assume that there exists some $\theta > \vartheta$ such that $\Lambda_\vartheta = \Lambda_\theta$.
Then, for any characteristic $f:\N_0\to\R^p$ with
	\begin{equation}	\label{eq:sum e^-theta n f(n) < infty}
		\sum_{n=0}^\infty f(n) e^{-\theta n} < \infty
	\end{equation}
	there exist vectors $b_{\lambda,k,f} \in \R^{p}$,
$k=0,\ldots,\k(\lambda)-1$, $\lambda \in \Lambda_\theta$ such that 
\begin{equation}	\label{eq:asymptotic expansion N_t lattice}
F(n) = \sum_{\lambda\in\Lambda_\theta}e^{\lambda n} \sum_{k=0}^{\k(\lambda)-1} n^k b_{\lambda,k,f}+O(e^{\theta n})
\quad	\text{as } n \to \infty
\end{equation}
where $F$ is the unique locally bounded solution to the discrete Markov renewal equation \eqref{eq:Markov renewal}.
\end{theorem}

Theorem \ref{Thm:Markov renewal theorem F(n) lattice} is the lattice analog of Theorem \ref{Thm:Markov renewal theorem F(t) non-lattice}.
The proof of the latter is based on the asymptotic expansion of the renewal function $\bU(t)$.
In the lattice case, we do not take this detour.
Nevertheless, for the sake of completeness, we provide the expansion of the renewal density in the lattice case.

\begin{corollary}	\label{Cor:Markov renewal theorem lattice}
Let $\bmu=(\mu^{i,j})_{i,j\in\p}$ be a $p \times p$ matrix
of locally finite measures concentrated on $\N_0$ such that (A\ref{ass:first moment}) and (A\ref{ass:subcritical instant offspring}) hold.
Then, there exist matrices $C_{\lambda,k} \in \R^{p \times p}$,
$k=0,\ldots,\k(\lambda)-1$, $\lambda \in \Lambda_\vartheta$ such that
\begin{equation*}	
\bU(\{n\}) = \sum_{\lambda\in\Lambda_\theta} e^{\lambda n} \sum_{k=0}^{\k(\lambda)-1} n^k C_{\lambda,k}+O(e^{\vartheta n})
\quad	\text{as } n \to \infty.
\end{equation*}
\end{corollary}

\subsection{Discussion of the assumptions and comparison with the literature}	\label{subsec:literature}

We conclude this section with a discussion of the assumptions we use and,
in particular, a comparison with the assumptions of the work of Sgibnev \cite{Sgibnev:2002}.
Regarding the latter, notice that our assumption (A\ref{ass:subcritical instant offspring})
is also a standing assumption in \cite{Sgibnev:2002}.
Moreover, Sgibnev works in the non-lattice case only and assumes that $\cL\bmu(0)$ is irreducible.
We believe that these two conditions serve the purpose of convenience,
but are not fundamentally necessary for the method of proof he uses, see e.g.\ \cite[Remark 3]{Sgibnev:2002}
regarding the extension of his result to the lattice case.
Further, Sgibnev makes an assumption that is analogous to our assumption (A\ref{ass:first moment}),
namely, the finiteness of a $\varphi$-moment of $\bmu$ for some semimultiplicative function $\varphi$
such that $\varphi(x)e^{-r_+x}$ is nondecreasing on $[0,\infty)$ for some $r_+ \geq 0$.
This is close to validity of (A\ref{ass:first moment}) with $\vartheta = -r_0 \leq 0$, which is a quite strong moment assumption.
However, we think that Sgibnev's requirement that $\vartheta \leq 0$ in (A\ref{ass:first moment})
can be removed with moderate effort, for example with the help of an exponential tilting.

Finally, Sgibnev's main result  \cite[Theorem 3]{Sgibnev:2002} requires, in our notation, that the spectral radius of the matrix
$\cL(\bmu^{*m})_{\mathsf s}(\vartheta)$ is strictly smaller than $1$ for some convolution power $m \in \N$
where here and throughout this section, we write $\nu_{\mathsf s}$ for the singular part of a measure $\nu$
on the Borel sets of $[0,\infty)$ and extend this canonically to matrices of measures by applying the singular part entrywise.
This assumption, which appears in Lemma \ref{Lem:comparison of conditions} as Condition (\ref{condition:A}),
must be seen in comparison to our assumption
\begin{align*}	
\sup_{\eta\in\R} \|(\Ip-\cL\bmu(\vartheta+\imag\eta))^{-1}\| < \infty,
\end{align*}
which follows from Condition (\ref{condition:E}) from Lemma \ref{Lem:comparison of conditions} as long as $\Real(\lambda)\not=\vartheta$ for any $\lambda\in\Lambda$.
In a recent paper by Iksanov and the first two authors \cite{Iksanov+al:2024},
as well as a paper by Janson and Neininger \cite{Janson+Neininger:2008}, both of which deal with the single-type case,
yet another condition is used, namely, $\limsup_{\eta \to \infty} |\cL\bmu(\vartheta+\imag\eta)|<1$.
A slightly weaker (multi-type) version of which appears in Lemma \ref{Lem:comparison of conditions} as Condition (\ref{condition:C}).
The following lemma provides a comparison of all these and other conditions.

\begin{lemma}	\label{Lem:comparison of conditions}
Suppose that (A\ref{ass:first moment}) and (A\ref{ass:subcritical instant offspring}) hold. Consider the following conditions.
\begin{enumerate}[(A)]
		\item	There is an $m \in \N$ such that $\rho_{\cL(\bmu^{*m})_{\mathsf s}(\vartheta)}<1$.
		\label{condition:A}
		\item	There is an $m \in \N$ such that $\|\cL(\bmu^{*m})_{\mathsf s}(\vartheta)\|<1$.
		\label{condition:B}
		\item	There is an $m \in \N$ such that $\limsup_{\eta \to \infty} \|\cL\bmu(\vartheta+\imag\eta)^{m}\|<1$.
		\label{condition:C}
		\item	There is an $m \in \N$ such that $\limsup_{\eta \to \infty} \sup_{\theta \geq \vartheta} \|\cL\bmu(\theta+\imag\eta)^{m}\|<1$.
		\label{condition:D}
		\item	There is an $\eta_0 \geq 0$ such that $\sup_{\eta \geq \eta_0}\|(\Ip-\cL\bmu(\vartheta+\imag\eta))^{-1}\|<\infty$.	
		\label{condition:E}
		\item	There is $\eta_0 \geq 0$ such that $\sup_{\eta \geq \eta_0, \theta \geq \vartheta}\|(\Ip-\cL\bmu(\theta +\imag\eta))^{-1}\|<\infty$.
		\label{condition:F}
\end{enumerate}
Then (\ref{condition:A}) is equivalent to (\ref{condition:B}), (\ref{condition:B}) implies (\ref{condition:C}),
(\ref{condition:C}) is equivalent to (\ref{condition:D}), each of the conditions (\ref{condition:C}), (\ref{condition:D}) and (\ref{condition:F}) implies (\ref{condition:E}). If, additionally, $\Lambda_\vartheta$ is finite, then (\ref{condition:E}) and (\ref{condition:F}) are equivalent.
Finally, (\ref{condition:D}) implies that $\Lambda_\vartheta$ is finite.
\end{lemma}
\begin{proof}
Since $\rho_{\cL(\bmu^{*m})_{\mathsf s}(\vartheta)} \leq  \|\cL(\bmu^{*m})_{\mathsf s}(\vartheta)\|$, see e.g.\ \cite[Section 10.3, Theorem 1]{Lancaster+Tismenetsky:1985},
(\ref{condition:B}) implies (\ref{condition:A}).
To see that the converse implication also holds true,
first observe that $(\bnu^{*n})_{\mathsf s} \leq (\bnu_{\mathsf s})^{*n}$ for any matrix $\bnu$ of measures on the Borel sets of $[0,\infty)$
and $n \in \N$.
Indeed, denoting by $\bnu_{\mathsf a}$ the absolutely continuous part of $\bnu$,
we have $\bnu^{*n}=(\bnu_{\mathsf a}+\bnu_{\mathsf s})^n=(\bnu_{\mathsf s})^{*n}+\bnu'$,
where $\bnu'$ is a sum of convolution products of length $n$ of measures $\bnu_{\mathsf a}$
and $\bnu_{\mathsf s}$ with at least one convolution factor $\bnu_{\mathsf a}$.
In particular, $\bnu'$ is absolutely continuous. This implies
\begin{align}	\label{eq:power of singular}
(\bnu^{*n})_{\mathsf s} = ((\bnu_{\mathsf s})^{*n}+\bnu')_{\mathsf s} = ((\bnu_{\mathsf s})^{*n})_{\mathsf s} \leq (\bnu_{\mathsf s})^{*n}.
\end{align} 
Let $m \in \N$ be such that the condition (\ref{condition:A}) holds.
From \eqref{eq:Gelfand} we infer the existence of some $n \in \N$ with $\|(\cL\bmu^{*m}_{\mathsf s}(\vartheta))^n\|<1$.
Set $\bnu=\bmu^{*m}$ in \eqref{eq:power of singular} to infer
\begin{align*}
((\bmu^{*m})_{\mathsf s})^{*n} \geq ((\bmu^{*m})^{*n})_{\mathsf s}=(\bmu^{*mn})_{\mathsf s}.
\end{align*}
In particular, $\|\cL(\bmu^{*mn})_{\mathsf s}(\vartheta)\| \leq \|(\cL\bmu^{*m}_{\mathsf s}(\vartheta))^n\|<1$
proving that (\ref{condition:B}) is satisfied.	\smallskip

\noindent
Now assume that (\ref{condition:B}) holds, i.e., $\|\cL(\bmu^{*m})_{\mathsf s}(\vartheta)\|<1$ for some $m \in \N$. Then
\begin{align*}
\|\cL\bmu(\vartheta+\imag\eta)^{m}\|=\|\cL\bmu^{*m}(\vartheta+\imag\eta)\|
\leq \|\cL(\bmu^{*m})_{\mathsf s}(\vartheta+\imag\eta)\| + \|\cL(\bmu^{*m})_{\mathsf a}(\vartheta+\imag\eta)\|.
\end{align*}
The Riemann-Lebesgue lemma yields $\cL(\bmu^{*m})_{\mathsf a}(\vartheta+\imag\eta) \to 0$ as $|\eta|\to\infty$,
which in turn gives
\begin{align*}
\limsup_{\eta\to \pm \infty}\|\cL\bmu(\vartheta+\imag\eta)^m\|
\leq \limsup_{\eta \to \pm \infty} \|\cL(\bmu^{*m})_{\mathsf s}(\vartheta+\imag\eta)\|
\leq \|\cL(\bmu^{*m})_{\mathsf s}(\vartheta)\| <1.
\end{align*}

\noindent
To prove the equivalence of (\ref{condition:C}) and (\ref{condition:D}), it is sufficient to prove the implication from (\ref{condition:C}) to (\ref{condition:D}) since (\ref{condition:D}) is formally stronger than (\ref{condition:C}).
Since the argument is somewhat more involved, it is moved to Lemma \ref{Lem:limsup||Lmu(theta+ieta)||<1} below.	\smallskip

\noindent
We continue with the proof that (\ref{condition:C}) implies (\ref{condition:E}).
Notice that validity of (\ref{condition:C}) implies the existence of $\eta_0 \geq 0$ and $0<\delta<1$
such that, for $|\eta| \geq \eta_0$, $\|\cL\bmu(\vartheta+\imag\eta)^m\| \leq 1-\delta$.
For any such $\eta$, the von Neumann formula yields
\begin{align*}
\|(\Ip-\cL\bmu(\vartheta+\imag\eta))^{-1}\|
&\leq \bigg\|\sum_{n=0}^\infty \cL\bmu(\vartheta+\imag\eta)^n\bigg\|
\leq \sum_{k=0}^{\infty}\sum_{l=0}^{m-1} \|\cL\bmu(\vartheta+\imag\eta)^m\|^k\|\cL\bmu(\vartheta+\imag\eta)\|^l	\\
&\leq	\frac1\delta \sum_{l=0}^{m-1}\|\cL\bmu(\vartheta)\|^l
\end{align*}
proving that also condition (\ref{condition:E}) holds.	\smallskip

Since (\ref{condition:F}) is formally stronger than (\ref{condition:E}),
we continue with the proof that (\ref{condition:E}) implies (\ref{condition:F})
under the additional assumption that $\Lambda_\vartheta$ is finite.
To this end, observe that if $\Lambda_\vartheta$ is finite, then we can find $\eta_0 \geq 0$
such that the function $z\mapsto (\Ip-\cL\bmu(z))^{-1}$ is holomorphic on a neighborhood of
$\{\Real(z) \geq \vartheta,\, \Imag(z) \geq \eta_0\}$.
Since $\cL\bmu(\theta+\imag\eta_0) \to \bmu(0)$ as $\theta \to \infty$ by the dominated convergence theorem
and since $\rho_{\bmu(0)}<1$, we conclude that 
\begin{align*}
\sup_{\theta \geq \vartheta} \big\| (\Ip-\cL\bmu(\theta+\imag\eta_0))^{-1} \big\| <\infty.
\end{align*}
Condition (\ref{condition:E}) ensures that the holomorphic function $z\mapsto (\Ip-\cL\bmu(z))^{-1}$ is bounded
on the boundary of the quadrant $\{\Real(z) \geq \vartheta,\, \Imag(z) \geq \eta_0\}$.
The function is therefore bounded on this set.

Finally, suppose that (\ref{condition:D}) holds. We prove that $\Lambda_\vartheta$ is then finite.
It suffices to prove that $\Lambda_\vartheta$ is contained in a compact rectangle,
which in turn is contained in $\overline{\cH}_{\vartheta} \defeq \{z \in \C:\Real(z) \geq \vartheta\}$.
\begin{center}    
\begin{tikzpicture}[scale=1]

\draw[->] (-2,0) -- (3,0) node[right] {$\Real(z)$};
\draw[->] (0,-2) -- (0,3) node[above] {$\Imag(z)$};

\fill[gray!20] (0.3,0) rectangle (2.6,1.6);
\fill[gray!20] (0.3,-1.6) rectangle (2.6,0);

\draw[black, densely dotted] (0.75,1.7) -- (0.88,2.1);
\draw[black, densely dotted] (1.05,1.7) -- (1.18,2.1);
\draw[black, densely dotted] (1.35,1.7) -- (1.48,2.1);
\draw[black, densely dotted] (1.65,1.7) -- (1.78,2.1);
\draw[black, densely dotted] (1.95,1.7) -- (2.08,2.1);
\draw[black, densely dotted] (2.25,1.7) -- (2.38,2.1);

\draw[black, densely dotted] (0.15,1.7) -- (0.57,3.1);
\draw[black,densely dotted] (0.45,1.7) -- (0.87,3.1);

\draw[black, densely dotted] (0.98,2.5) -- (1.17,3.1);
\draw[black, densely dotted] (1.28,2.5) -- (1.47,3.1);
\draw[black, densely dotted] (1.58,2.5) -- (1.77,3.1);
\draw[black, densely dotted] (1.88,2.5) -- (2.07,3.1);
\draw[black, densely dotted] (2.18,2.5) -- (2.37,3.1);
\draw[black, densely dotted] (2.48,2.5) -- (2.67,3.1);
\draw[black, densely dotted] (2.54,1.7) -- (2.97,3.1);

\draw[black, densely dotted] (2.6,0) -- (3.57,3.1);
\draw[black, densely dotted] (2.9,0) -- (3.87,3.1);
\draw[black, densely dotted] (3.34,0.4) -- (4.17,3.1);
\draw[black, densely dotted] (3.64,0.4) -- (4.47,3.1);

\draw[black, densely dotted] (2.84,1.7) -- (3.27,3.1);
\draw[black, densely dotted] (2.84,1.7) -- (2.6,0.9);

\draw[black] (0.3,-0.4) -- (0.3,1.6);
\draw[black] (0.3,-0.9) -- (0.3,-1.6);
\draw[black] (2.6,-0.4) -- (2.6,1.6);
\draw[black] (2.6,-0.9) -- (2.6,-1.6);
\draw (0.3,-.1) -- (0.3,.1) node[below=14pt] {$\vartheta$};
\draw (2.6,-.1) -- (2.6,.1) node[below=14pt] {$\theta$};
\draw (-.1,1.6) -- (0.1,1.6) node[black, left=12pt] {$\eta_0$};
\draw[black, thick] (-0.1,1.6) -- (3,1.6);
\draw (-.1,-1.6) -- (0.1,-1.6) node[black, left=12pt] {$-\eta_0$};
\draw[black, thick] (-0.1,-1.6) -- (3.1,-1.6);
\draw[black, densely dotted] (2.65,-1.6) -- (3.1,-0.2);
\draw[black, densely dotted] (2.95,-1.6) -- (3.4,-0.2);
\draw[black, densely dotted] (0.1,-2) -- (0.2,-1.7);
\draw[black, densely dotted] (0.4,-2) -- (0.5,-1.7);
\draw[black, densely dotted] (0.7,-2) -- (0.8,-1.7);
\draw[black, densely dotted] (1,-2) -- (1.1,-1.7);
\draw[black, densely dotted] (1.3,-2) -- (1.4,-1.7);
\draw[black,densely dotted] (1.6,-2) -- (1.7,-1.7);
\draw[black,densely dotted] (1.9,-2) -- (2,-1.7);
\draw[black,densely dotted] (2.2,-2) -- (2.3,-1.7);
\draw[black,densely dotted] (2.5,-2) -- (2.6,-1.7);
\draw[black,densely dotted] (2.8,-2) -- (2.9,-1.7);
\node [black] at (1.67,2.3) {$\|\cL\bmu(z)\|<1$};

\label{Fig:Lambda_theta in rectangle}
\end{tikzpicture}


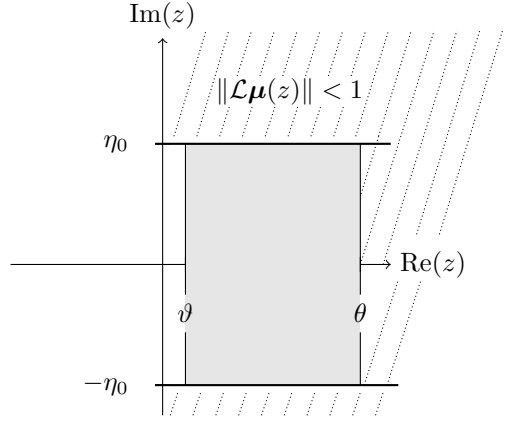
\captionof{figure}{$\Lambda_{\vartheta}$ is contained in a compact rectangle in the right half-plane.}
\end{center}
Indeed, the entries of $\cL\bmu(z)$ are holomorphic on $\overline{\cH}_{\vartheta}$,
and hence so is $\det(\Ip-\cL\bmu(z))$.
Further, $\det(\Ip-\cL\bmu(z))$ is not identically zero by Proposition \ref{Prop:properties of varrho}.
Hence the set of its zeroes has no accumulation point within $\overline{\cH}_{\vartheta}$ and
in particular not within the compact rectangle.

To prove the existence of such a compact rectangle,
first recall that $\det(\Ip-\cL\bmu(\lambda))=0$ implies that for every $m$, $\cL\bmu^m(\lambda)$ has eigenvalue $1$ 
and, as a consequence, $\|(\cL\bmu(\lambda))^m\| \geq 1$.
On the other hand, $\rho_{\cL\bmu(\theta)} \to \rho_{\bmu(0)} < 1$ as $\theta \to \infty$
by Proposition \ref{Prop:properties of varrho}(b) and (A\ref{ass:subcritical instant offspring}).
Consequently, $\|(\cL\bmu(z))^m\|<1$ for some $m$ and $\Real(z)\ge\theta$ for large enough $\theta$ proving $\Lambda\subseteq\{z:\Real(z)\le\theta\}$. 
Further,  from \eqref{condition:D} and the fact that $\cL\mu(\bar{z})=\overline{\cL\mu(z)}$ follows that $\Lambda\subseteq\{z:\Imag(z)\le\eta_0\}$, for large enough $\eta_0$.
\end{proof}

\begin{lemma}	\label{Lem:limsup||Lmu(theta+ieta)||<1}
Suppose that (A\ref{ass:first moment}) and (A\ref{ass:subcritical instant offspring}) hold
and let $m \in \N$. If
\begin{align}	\label{eq:limsup||Lmu(theta+ieta)||<1}
\limsup_{\eta\to\pm\infty} \| \cL\bmu(\vartheta+\imag\eta)^m \|<1,
\end{align}
then
\begin{align}	\label{eq:matrix norm uniformly small}
\limsup_{\eta \to \pm\infty} \sup_{\theta \geq \vartheta} \|\cL\bmu(\theta + \imag \eta)^m\|<1.
\end{align}
\end{lemma}

The proof is an extension of that of Lemma 2.1 in \cite{Janson+Neininger:2008}.
Even though the proof of the latter result carries over without major obstacles, we give a version adapted to the
situation here for the reader's convenience.

\begin{proof}
From (A\ref{ass:first moment}), we have
$\overline{\cH}_{\vartheta} = \{z \in \C:\Real(z) \geq \vartheta\} \subseteq \interior{\Dom(\cL\bmu(\vartheta))}$.
Now suppose that, for some $m \in \N$, \eqref{eq:limsup||Lmu(theta+ieta)||<1} holds.
Choose $\varepsilon > 0$ with
\begin{align*}	\textstyle
\limsup_{\eta \to \infty} \|\cL\bmu(\vartheta+\imag\eta)^m\| < 1-2\varepsilon
\end{align*}
and $\zeta \geq 0$ with $\|\cL\bmu(\vartheta+\imag\eta)^m\| \leq 1-2\varepsilon$ when $|\eta| \geq \zeta$.
Define $C \defeq \sup_{\eta \in \R} \|\cL\bmu(\vartheta+\imag\eta)^m\|$ and notice that $C<\infty$.
The entries of $(\cL\bmu)^m$ are Laplace transforms of locally finite measures on $[0,\infty)$
and all bounded and continuous on $\overline{\cH}_{\vartheta}$, and holomorphic, thus harmonic, on a neighborhood of $\overline{\cH}_{\vartheta}$.
We may thus apply the Poisson transformation on the closed half-space $\overline{\cH}_\vartheta$
and infer that $(\cL\bmu)^m$ is the integral of its boundary values with respect to the Poisson kernel
$\Psf_{\!x}(y) \defeq \frac{1}{\pi} \frac{x}{x^2+y^2}$ for $x,y\in\R$.
More precisely,
\begin{align}	\label{eq:Poisson tranform}
\cL\bmu(\theta + \imag \eta)^m = \int \limits_{-\infty}^{\infty} \cL\bmu(\vartheta+\imag y)^m \, \Psf_{\theta-\vartheta}(\eta-y) \, \dy, \quad \theta > \vartheta, \, \eta \in \R,
\end{align}
see e.g.\ \cite[Lemma 3.4 on p.\;17]{Garnett:1981}.
We conclude
\begin{align}
\|\cL\bmu(\theta + \imag \eta)^m\|
&= \bigg\|\int \limits_{-\infty}^{\infty} \cL\bmu(\vartheta+\imag y)^m \, \Psf_{\theta-\vartheta}(\eta-y) \, \dy\bigg\|
\leq \int \limits_{-\infty}^{\infty} \big\|\cL\bmu(\vartheta+\imag y)^m\| \, \Psf_{\theta-\vartheta}(\eta-y) \, \dy	 \notag	\\
&= \int \limits_{\{|y|>\zeta\}} \!\!\! \big\|\cL\bmu(\vartheta+\imag y)^m\| \, \Psf_{\theta-\vartheta}(\eta-y) \, \dy	
+\int \limits_{-\zeta}^\zeta \big\|\cL\bmu(\vartheta+\imag y)^m\| \, \Psf_{\theta-\vartheta}(\eta-y) \, \dy	\notag	\\
&\leq 1-2\varepsilon + C \int_{-\zeta}^\zeta \Psf_{\theta-\vartheta}(\eta-y) \, \dy.	\label{eq:bound from Poisson trafo}
\end{align}
Define
\begin{equation*}
\omega(\theta + \imag \eta) \defeq \int_{-\zeta}^\zeta \!\!\! \Psf_{\theta-\vartheta}(\eta-y) \, \dy \geq 0.
\end{equation*}
The set $\cH \defeq \{z\in\cH_{\vartheta}:\omega(z)>\frac{\varepsilon}{C}\}$ is an intersection of the right half-plane $\cH_{\vartheta}$
and a circular disk, see \cite[p.\;13]{Garnett:1981}. In particular, it is a bounded set.
Thus, $H' \defeq \sup\{\Imag(z):z\in\cH \}<\infty$.
Now, if $\Real(z)\geq \theta$ and $|\Imag(z)|>H'$, then $\omega(z) \leq \frac{\varepsilon}{C}$.
Hence, 
\begin{align*}
\|\cL\bmu(\theta + \imag \eta)^m\| \leq 1-\varepsilon
\end{align*}
proving that \eqref{eq:matrix norm uniformly small} holds.
\end{proof}

\section{Proofs}

\subsection{Spectral radius and Malthusian parameter}	\label{subsec:varrho}

We first collect all the basic results about the spectral radius of the matrix $\cL\bmu$
of the Laplace transforms.
In order to simplify the notation we define the function
\begin{align}	\label{eq:varrho}
\varrho: \Dom(\cL\bmu) \to [0,\infty),	\quad	z \mapsto \varrho(z) \defeq \rho_{\cL \bmu(z)}.
\end{align}

\begin{proof}[Proof of Proposition \ref{Prop:largest root}]
Suppose that (A\ref{ass:first moment}) and (A\ref{ass:subcritical instant offspring})  hold,
in particular, $\cL\bmu(\vartheta)$ has finite entries only for some $\vartheta \in \R$.	\smallskip

\noindent
(a)
First suppose there is some $\alpha \in \Lambda \cap \R$ with $\rho_{\cL\bmu(\alpha)} = 1$.
By (A\ref{ass:subcritical instant offspring}) and Proposition \ref{Prop:properties of varrho}, we have
$\alpha \in \{\theta \in \Dom(\cL\bmu) \cap \R: \varrho(\theta) > \rho_{\bmu(0)}\}$
and, therefore, that $\varrho|_{\Dom(\cL\bmu) \cap \R}$ is decreasing
and even strictly decreasing in a neighborhood of $\alpha$ in $\Dom(\cL\bmu)$.
In particular, $\alpha$ is unique with the property $\varrho(\alpha) = 1$.	\smallskip

\noindent
(b) 
Now suppose that additionally (A\ref{ass:Lmu(vartheta) geq 1}) holds, i.e.,
there is $\vartheta \in \R$ with $1 \leq \varrho(\vartheta) < \infty$.
By Proposition \ref{Prop:properties of varrho}, $\varrho$
is continuous and decreasing with $\lim_{\theta \to \infty} \varrho(\theta) = \rho_{\bmu(0)}<1$,
so by the intermediate value theorem there exists an $\alpha \geq \vartheta$
with $\varrho(\alpha)=1$.
The matrix $\cL\bmu(\alpha)$ thus has eigenvalue $1$
and, therefore, $\det(\Ip-\cL\bmu(\alpha))=0$, i.e., $\alpha \in \Lambda \cap \R$.	\smallskip

\noindent
(c)
If $\Lambda$ is non-empty,
then there exists $\lambda \in \Dom(\cL\bmu)$ with $\det(\Ip-\cL\bmu(\lambda))=0$,
that is, $1$ is an eigenvalue of $\cL\bmu(\lambda)$, so $\varrho(\lambda) \geq 1$.
By Proposition \ref{Prop:properties of varrho}(d), $\varrho(\Real(\lambda))\geq 1$,
so (A\ref{ass:Lmu(vartheta) geq 1}) holds.	\smallskip

\noindent
(d) Suppose that $\bmu$ is primitive.
Let $\theta > \alpha$. Then $\varrho(\theta)<1$ and
is the Perron-Frobenius eigenvalue of $\cL\bmu(\theta)$.
By Lemma \ref{Lem:largest eigenvalue cts and mon},
there exists a unique eigenvector $v_\theta$ associated with the eigenvalue $\varrho(\theta)$
of $\cL\bmu(\theta)$ with positive entries only and $|v_\theta|=1$.
Moreover, the lemma yields that $\theta \mapsto v_\theta$ is continuous.
In particular, all entries of $v_\theta$ are bounded away from $0$
in a right neighborhood $[\alpha,\alpha+\epsilon]$ of $\alpha$
and for any given vector $y \in \Rnn^d$, we can find
a finite constant $c>0$ such that $y \leq c v_\theta$ for every $\theta \in [\alpha,\alpha+\epsilon]$.
For $\alpha < \theta \leq \alpha+\epsilon$, using Lemma \ref{Lem:rho_A<1}, we infer
\begin{align*}
(\theta-\alpha) (\Ip - \cL\bmu(\theta))^{-1} y
&= (\theta-\alpha) \sum_{n=0}^\infty \cL\bmu(\theta)^n y
\leq c (\theta-\alpha) \sum_{n=0}^\infty \cL\bmu(\theta)^n v_\theta	\\
&= c (\theta-\alpha) \sum_{n=0}^\infty \varrho(\theta)^n v_\theta
= c\frac{\theta-\alpha}{1-\varrho(\theta)} v_\theta
\leq \frac{c \epsilon}{1-\varrho(\alpha+\epsilon)} v_\theta
\end{align*}
by the convexity of $\varrho(\theta)$. The last expression remains bounded as $\theta \downarrow \alpha$.
Therefore, $\alpha$ is a pole of order at most one for each of the entries of $(\Ip-\cL\bmu(z))^{-1}$
and a pole of order one for at least one of these entries.
Further, suppose that $\lambda \in \C$ satisfies $\Real(\lambda)=\alpha$, i.e., $\lambda = \alpha + \imag \eta$ for some $\eta \not = 0$.
We may assume without loss of generality that $\eta>0$.
Then each entry of $\cL\bmu(\lambda)$ is bounded in absolute value
by the corresponding entry of $\cL\bmu(\alpha)$, so $\varrho(\lambda) \leq \varrho(\alpha) = 1$
by Lemma \ref{Lem:rho is monotone}. By the Perron-Frobenius theorem \cite[Theorem 1.1]{Seneta:1981}, equality can only hold
if the absolute value of every entry of $\cL\bmu(\lambda)$ equals the corresponding entry of $\cL\bmu(\alpha)$.
This means that $\cL\mu^{i,j}(\alpha+\imag \eta) = \cL\mu^{i,j}(\alpha)$ for all $i,j\in\p$.
Hence, $\mu^{i,j}$ is concentrated on $2\pi \Z/\eta$. Consequently, we are in the lattice case.
By assumption, the smallest lattice on which all $\mu^{i,j}$ are concentrated is $\Z$,
so $\eta \geq 2\pi$ and hence $\lambda \not \in \Lambda$ by the definition of $\Lambda$ in the lattice case.
\end{proof}

\subsection{Existence and uniqueness}	\label{subsec:existence and uniqueness}

\begin{proof}[Proof of Proposition \ref{Prop:existence and uniqueness MRE}]
(b)
By assumption (A\ref{ass:first moment}), we have $\cL\bmu(\vartheta) < \infty$ for some $\vartheta \in \R$.
Recall from \eqref{eq:rho_mu(theta)->rho_mu(0)} that
$\rho_{\cL\bmu(\theta)} \to \rho_{\bmu(0)} < 1$ as $\theta \to \infty$,
where the last inequality is guaranteed by (A\ref{ass:subcritical instant offspring}).
Pick $\theta \in \R$ with $\rho_{\cL\bmu(\theta)}<1$. Then
\begin{align}	\label{eq:LU(theta)<infty}
\cL \bU(\theta)
= \sum_{n=0}^\infty \, \cL\bmu^{*n}(\theta) = \sum_{n=0}^\infty \, \cL\bmu(\theta)^n = (\Ip-\cL\bmu(\theta))^{-1},
\end{align}
which is a finite matrix by Lemma \ref{Lem:rho_A<1}.
Lemma \ref{Lem:growth of fcts with finite LT} now gives $U^{i,j}(t) = o(e^{\theta t})$ as $t \to \infty$
for $i,j\in\p$.	\smallskip

(a)
The case where (A\ref{ass:subcritical instant offspring}) holds but (A\ref{ass:first moment}) is violated
can be reduced to the case where (A\ref{ass:first moment}) and (A\ref{ass:subcritical instant offspring}) hold.
Indeed, fix $t > 0$ und replace $\bmu(\cdot)$ by $\bmu_t(\cdot) \defeq \bmu(\cdot \cap [0,t])$.
The Laplace transform of $\bmu_t(\cdot)$ is finite everywhere and the associated Markov renewal measure
coincides with $\bU$ on $[0,t]$. In particular, $\bU(t)$ is finite.	\smallskip

\noindent
(c)
The assertion can be obtained from the proof of \cite[Theorem 2.1]{Crump:1970a}.
While the cited theorem covers only the irreducible case, the proof also works without this assumption. 
%
\end{proof}

\subsection{Asymptotic expansions in the non-lattice case}	\label{subsec:asymptotic expansions non-lattice}

\begin{proof}[Proof of Theorem \ref{Thm:Markov renewal theorem U(t) non-lattice}]
We first suppose that $\vartheta > 0$.
Recall that the (matrix-valued) renewal function $\bU$ satisfies the renewal equation
\begin{align}	\tag{\ref{eq:renewal for U}}
(U^{i,j}(t))_{i,j \in \p} = \bU(t) = \1_{[0,\infty)}(t) \Ip + \bmu*\bU(t),	\quad	t \in \R.
\end{align}
As in \cite{Iksanov+al:2024}, the strategy is to take Laplace transforms in \eqref{eq:renewal for U},
to analyze these and to use the Laplace inversion formula
to infer the asymptotic behavior of $\bU$.
However, the Laplace inversion theorem requires certain integrability properties which are
satisfied only when using a smoothed version of the indicator $\1_{[0,\infty)}$.
Therefore, for every $\varepsilon>0$,
we define $g_\varepsilon(t) \defeq \1_{[0,\varepsilon)}(t) \frac1\varepsilon t + \1_{[\varepsilon,\infty)}(t)$ for $t \in \R$.
An alternative representation of $g_\varepsilon$ is $(\frac1\varepsilon \1_{[0,\varepsilon]})*\1_{[0,\infty)}(t)$, $t \in \R$.
\begin{center}
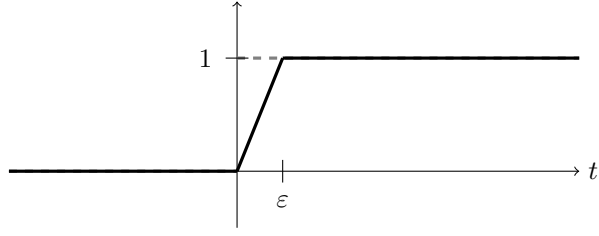

\begin{tikzpicture}[scale=1.5]	\label{fig:f_epsilon}
\draw[->] (-2,0) -- (3,0) node[right] {$t$};
\draw[->] (0,-0.5) -- (0,1.5);
\draw[very thick,gray, dashed] (0,1) -- (3,1);
\draw[very thick, gray, dashed] (-2,0) -- (0,0);
\draw[very thick] (0,0) -- (0.4,1);
\draw[very thick] (-2,0) -- (0,0);
\draw[very thick] (0.4,1) -- (3,1);
\draw (0.4,-.1) -- (0.4,.1) node[below=10pt] {$\varepsilon$};
\draw (-.1,1) -- (.1,1) node[left=10pt] {$1$};
\end{tikzpicture}
\captionof{figure}{Graphs of $\1_{[0,\infty)}$ (dashed) and $g_{\varepsilon}$ (solid).}
\end{center}
Clearly,
\begin{align*}
g_{\varepsilon}(t)  \leq \1_{[0,\infty)}(t) \leq g_{\varepsilon}(t+\varepsilon),	\quad	t \in \R.
\end{align*}
This immediately implies
\begin{align}	\label{eq:U_eps(t) leq U(t) leq U_eps(t+eps)}
\bU_{\!\varepsilon}(t) \defeq g_\varepsilon \Ip * \bU(t) \leq \bU(t) \leq \bU_{\!\varepsilon}(t+\varepsilon),	\quad	t \in \R,
\end{align}
where as usual the inequality has to be understood entrywise.
According to Lemma \ref{Lem:rho_A<1},
the Laplace transform
\begin{align*}
\cL \bU(z) \defeq \bigg(\int e^{-z x} \, \bU(\dx)\bigg)_{\!i,j\in\p}
= \sum_{n=0}^\infty \cL\bmu(z)^n
= (\Ip-\cL\bmu(z))^{-1}
\end{align*}
is finite for every $z$ with $\rho_{\cL\bmu(\Real(z))}<1$.
For any such $z$, we infer
\begin{align}
\cL \bU_{\!\varepsilon}(z)
&= \cL((\tfrac1\varepsilon \1_{[0,\varepsilon]})* \bU)(z)
= \cL((\tfrac1\varepsilon \1_{[0,\varepsilon]})* (\1_{[0,\infty)} \Ip + \bmu*\bU))(z)	\notag	\\
&= \cL((\tfrac1\varepsilon \1_{[0,\varepsilon]}) * \1_{[0,\infty)} \Ip)(z)
+ \cL((\tfrac1\varepsilon \1_{[0,\varepsilon]}) *\bmu*\bU )(z)	\notag  \\
&= \cL g_\varepsilon(z) \Ip
+ \cL(\bmu * \bU_{\!\varepsilon})(z)		\notag	\\	
&= \cL g_\varepsilon(z) \Ip
+ \cL\bmu(z) \cdot \cL \bU_{\!\varepsilon}(z)	\label{eq:Laplace matrix implicit}
\end{align}
where we have used
$(\tfrac1\varepsilon \1_{[0,\varepsilon]}) * \bmu * \bU
= \bmu * (\tfrac1\varepsilon \1_{[0,\varepsilon]}) * \bU = \bmu * \bU_{\!\varepsilon}(z)$
and the convolution theorem for Laplace transforms.
Solving \eqref{eq:Laplace matrix implicit} for $\cL \bU_{\varepsilon}(z)$
gives
\begin{align}	\label{eq:Laplace matrix explicit}
\cL \bU_{\varepsilon}(z) = (\Ip-\cL\bmu(z))^{-1} \cL g_\varepsilon(z)
\end{align}
provided that $(\Ip-\cL\bmu(z))$ is invertible, i.e.\ $\det(\Ip-\cL\bmu(z)) \not = 0$.
Now use that
\begin{align*}
\cL g_\varepsilon(z) = \cL(\tfrac1\varepsilon \1_{[0,\varepsilon]} * \1_{[0,\infty)})(z)
= \cL(\tfrac1\varepsilon \1_{[0,\varepsilon]})(z) \cdot \cL\1_{[0,\infty)}(z)
= \frac{1-e^{-\varepsilon z}}{\varepsilon z} \cdot \frac{1}{z} = \frac{1-e^{-\varepsilon z}}{\varepsilon z^2}
\end{align*}
to infer
\begin{align}	\label{eq:Laplace matrix even more explicit}
\cL \bU_{\!\varepsilon}(z) = \frac{1-e^{-\varepsilon z}}{\varepsilon z^2} (\Ip-\cL\bmu(z))^{-1}
\end{align}
for $z \in \Dom(\cL\bmu)$ with $\det(\Ip-\cL\bmu(z)) \not = 0$.
The right-hand side of \eqref{eq:Laplace matrix even more explicit}
defines a meromorphic extension of $\cL \bU_{\!\varepsilon}(z)$ on the half-plane $\Real(z)>\vartheta$.
If any of the entries of the matrix on the right-hand side of \eqref{eq:Laplace matrix even more explicit}
has a pole in $\lambda \in \C$ with $\Real(\lambda)>\vartheta$, then $\lambda \in \Lambda_\vartheta$.
Recall from Lemma~\ref{Lem:comparison of conditions}\eqref{condition:F} that
\begin{align}	\label{eq:sup over rectangles bound}
\sup_{\theta \geq \vartheta, |\eta| \geq \eta_0} \|(\Ip-\cL\bmu(\theta+\imag \eta))^{-1}\|<\infty
\end{align}
for some suitably large $\eta_0 > 0$.
In particular, the entries of $(\Ip - \cL\bmu(z))^{-1}$ are bounded on $\{z \in \C: \Real(z) \geq \vartheta, \, |\Imag(z)| \geq \eta_0\}$.
Therefore, the entries of the matrix on the right-hand side of \eqref{eq:Laplace matrix even more explicit}
are bounded in absolute value by a constant times $|z|^{-2}$ on $\{z \in \C: \Real(z) \geq \vartheta, \, |\Imag(z)| \geq \eta_0\}$,
making them integrable along vertical lines that do not contain a root.
Pick $\sigma > \vartheta$ with $\rho_{\cL\bmu(\sigma)} < 1$
(such a $\sigma$ exists by Proposition \ref{Prop:properties of varrho}(b)
and (A\ref{ass:subcritical instant offspring})). Then, in particular, $\sigma > \Real(\lambda)$ for all $\lambda \in \Lambda_{\vartheta}$
by Proposition \ref{Prop:properties of varrho}(d).
For any such $\sigma$, the Laplace inversion formula gives
\begin{align}	\label{eq:Laplace inversion}
\bU_{\!\varepsilon}(t) = \frac{1}{2\pi\imag}  \int \limits_{\sigma-\imag \infty}^{\sigma+\imag\infty} e^{t z}  \cL  \bU_{\!\varepsilon}(z) \, \dz.
\end{align}
The residue theorem then gives
\begin{align}
2\pi\imag  \sum_{\lambda\in\Lambda_{\vartheta}} \textrm{Res}_{z=\lambda}(e^{t z}  \cL \bU_{\!\varepsilon}(z))
&= \int_{\sigma-\imag R}^{\sigma+\imag R} e^{t z}   \cL  \bU_{\!\varepsilon}(z) \, \dz
+ \int_{\sigma+\imag R}^{\vartheta+\imag R} e^{t z}  \cL  \bU_{\!\varepsilon}(z) \, \dz \nonumber \\
&\hphantom{=} + \int_{\vartheta+\imag R}^{\vartheta-\imag R} e^{t z}  \cL  \bU_{\!\varepsilon}(z) \, \dz
+ \int_{\vartheta-\imag R}^{\sigma-\imag R} e^{t z}  \cL \bU_{\!\varepsilon}(z) \, \dz
\label{eq:application of residue theorem}
 \end{align}
for all sufficiently large $R > 0$.

\begin{center}
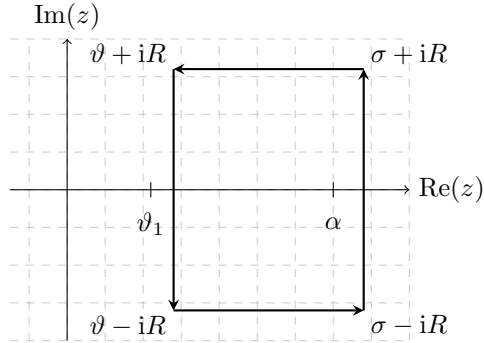

\begin{tikzpicture}[scale=1]
	\draw[gray!40, step=0.5,dashed] (-1.25,-2) grid (4,2);
	\draw[->] (-1.25,0) -- (4,0) node[right] {$\Real(z)$};
	\draw[->] (-0.5,-2) -- (-0.5,2) node[above] {$\Imag(z)$};

	\draw [stealth-,thick] (0.9,-1.6) -> (0.9,1.6);
	\draw [-stealth,thick] (0.9,-1.6) -> (3.4,-1.6);
	\draw [-stealth,thick] (3.4,-1.6) -> (3.4,1.6);
	\draw [-stealth,thick] (3.4,1.6) -> (0.9,1.6);

	\draw (.6,-.1) -- (.6,.1) node[below=8pt] {$\vartheta_1$};
	\draw (3,-.1) -- (3,.1) node[below=10pt] {$\alpha$};
	\node [black] at (4,1.8) {$\sigma+\mathrm{i}R$};
	\node [black] at (0.3,1.8) {$\vartheta+\mathrm{i}R$};
	\node [black] at (4,-1.8) {$\sigma-\mathrm{i}R$};
	\node [black] at (0.3,-1.8) {$\vartheta-\mathrm{i}R$};
\label{fig:Residue theorem}
\end{tikzpicture}
\captionof{figure}{A piecewise continuously differentiable path that encloses all solutions $\lambda\in\Lambda_\vartheta$.}
\end{center} 
By rearranging \eqref{eq:application of residue theorem}, we obtain
\begin{align}
\int_{\sigma-\imag R}^{\sigma+\imag R} e^{t z}  \cL  \bU_{\!\varepsilon}(z) \: \dz
&= 2\pi \imag \sum_{\lambda\in\Lambda_{\vartheta}} \text{Res}_{z=\lambda}(e^{t z} \cL \bU_{\!\varepsilon}(z))
+ \int_{\vartheta+\imag R}^{\sigma+\imag R} e^{t z} \cL \bU_{\!\varepsilon}(z) \: \dz	\notag	\\
&\hphantom{=} + \int_{\vartheta-\imag R}^{\vartheta+\imag R} e^{tz} \cL  \bU_{\!\varepsilon}(z) \: \dz
- \int_{\vartheta-\imag R}^{\sigma-\imag R} e^{t z} \cL \bU_{\!\varepsilon}(z) \: \dz.
\label{eq:residue thm rearranged}
 \end{align}
The first integral on the right-hand side of \eqref{eq:residue thm rearranged} can be estimated as follows
\begin{align}
\Bigl\lVert \int_{\vartheta+\imag R}^{\sigma+\imag R} e^{t z}   \cL  \bU_{\!\varepsilon}(z) \: \dz \Bigl\rVert
&\leq \int_{\vartheta+\imag R}^{\sigma+\imag R} |e^{t z}| \lVert  \cL  \bU_{\!\varepsilon}(z) \rVert \: \dz  \notag \\
&\leq \int_{\vartheta+\imag R}^{\sigma+\imag R} |e^{t z}| | \tfrac{1-e^{-\varepsilon z}}{\varepsilon z^2}| \|(\Ip-\cL \bmu(z))^{-1}\| \dz  \notag \\
&\leq C e^{t  \sigma} \int_{\vartheta}^{\sigma}  \big| \tfrac{1}{\varepsilon (x+\imag R)^2} \big|  \: \dx
\to 0	\quad	\text{as } R \to \infty	\label{eq:1st integral vanishes}
\end{align}
where $C > 0$ is a suitable finite constant the existence of which follows from \eqref{eq:sup over rectangles bound}.
The last integral on the right-hand side of \eqref{eq:residue thm rearranged} can be estimated analogously.
Taking $R\to\infty$, we thus infer
\begin{align}
\bU_{\varepsilon}(t)
&= \frac{1}{2\pi\imag} \int_{\sigma-\imag\infty}^{\sigma+\imag\infty} e^{t z}   \cL \bU_{\!\varepsilon}(z) \, \mathrm{d} z \notag	\\
&= \sum_{\lambda\in\Lambda_{\vartheta}} \text{Res}_{z=\lambda}(e^{t z}   \cL \bU_{\!\varepsilon}(z))
+ \frac{1}{2\pi\imag} \int_{\vartheta-\imag\infty}^{\vartheta+\imag\infty} e^{t z} \cL \bU_{\!\varepsilon}(z) \, \dz.
\label{eq:LT with inversion and residues}
\end{align}
We now focus on the sum on the right-hand side of \eqref{eq:LT with inversion and residues}.
Here, we expand each entry of the matrix $(\Ip-\cL\bmu(z))^{-1}$ into a Laurent series
around the point $z=\lambda\in\Lambda_{\vartheta}$.
Note that $\lambda$ is a root of $\det(\Ip-\cL\bmu(z))=0$. Because
\begin{align*}
(\Ip-\cL\bmu (z))^{-1}=\frac{1}{\det(\Ip-\cL\bmu(z))}\; \adj(\Ip-\cL\bmu(z)),
\end{align*}
each entry of $(\Ip-\cL\bmu(z))^{-1}$ has a pole with multiplicity at most
the multiplicity of the zero of $\det(\Ip-\cL\bmu(z))$ at $z=\lambda$.
Recall that $\k(\lambda)\in\N$ denotes the maximal multiplicity of the poles
of the entries of $(\Ip-\cL\bmu (z))^{-1}$ at $\lambda$.
Let us define $a^{ij}_{\lambda,k}$ as the coefficient in front of $(z-\lambda)^{-k}$ in the Laurent series of the $(i,j)$-th entry in $(\Ip-\cL\bmu(z))^{-1}$ around the point $\lambda\in\Lambda_{\theta}$.
In particular, $a^{ij}_{\lambda,k}=0$ for $k>\k(\lambda)$.
Summarizing, let us consider
\begin{align*}
A_{\lambda,k} \defeq (a^{ij}_{\lambda,k})_{i,j\in\p}
\end{align*}
as the matrix of all coefficients in front of $(z-\lambda)^{-k}$, see also \eqref{eq:expansion of I_p-Lmu inverse at lambda}.
Then, for $\lambda \in \Lambda_\vartheta$,
\begin{align}
\text{Res}_{z=\lambda}(e^{t z}   \cL \bU_{\!\varepsilon}(z))
&= \text{Res}_{z=\lambda}(e^{t z}  \cL g_{\varepsilon}(z)  (\Ip- \cL\bmu(z))^{-1}) \notag \\
&= e^{\lambda t} \sum\limits_{\substack{n,l\geq 0, \\ n+l<\k(\lambda)}} \frac{t^l}{l !}   \frac{1}{n !} \; \cL  g^{(n)}_{\varepsilon}(\lambda)  A_{\lambda,n+l+1}.
\label{eq:residue with epsilon}
\end{align}
Here, in order to arrive at \eqref{eq:residue with epsilon},
one has to determine the residue of $e^{t z}  \cL g_{\varepsilon}(z)  (\Ip-\cL\bmu(z))^{-1}$ at the point $z=\lambda$.
To this end, one can compute the corresponding Laurent series around $z=\lambda$,
as the residue corresponds to the coefficient with index $-1$.
This Laurent series can be obtained by multiplying the Laurent series of $e^{t z}$, $\cL g_{\varepsilon}(z)$,
and $(\Ip-\cL\bmu(z))^{-1}$ around $z=\lambda$.
We have
\begin{equation*}
e^{t z}=\sum_{l=0}^{\infty} \frac{t^l  e^{t\lambda}}{l!} \; (z-\lambda)^l
\quad	\text{and}	\quad
\cL g_{\varepsilon}(z)=\sum_{n=0}^{\infty}c_n \;(z-\lambda)^n
\end{equation*}
with coefficients $c_n=\frac{1}{n!}\cL g_{\varepsilon}^{(n)}(\lambda)$ for $n\in\N_0$ and $\lambda>0$.
The series above are then multiplied, and for the residue, we need the coefficient of $(z-\lambda)^{-1}$.
In particular, for $n+l\geq\k(\lambda)$, the coefficient matrix $A_{\lambda,n+l+1}$ is the zero matrix,
which means that only summands with indices with $n+l<\k(\lambda)$ contribute.

Write $g_{\varepsilon} = \1_{[0,\infty)} + g_{\varepsilon}-\1_{[0,\infty)}$ and use the linearity of the Laplace transform
to infer $\cL g_{\varepsilon} = \cL \1_{[0,\infty)} + \cL(g_{\varepsilon}-\1_{[0,\infty)})$,
we can split the right-hand side of \eqref{eq:residue with epsilon} as follows
\begin{align}
e^{\lambda  t} \!\!\!\!\!  \sum\limits_{\substack{n,l\geq 0, \\ n+l<\k(\lambda)}} \frac{t^l}{l !}   \frac{1}{n !}  \cL g^{(n)}_{\varepsilon}(\lambda) A_{\lambda,n+l+1}
&= e^{\lambda t} \!\!\!\!\! \sum\limits_{\substack{n,l\geq 0, \\ n+l<\k(\lambda)}} \frac{t^l}{l !}  \frac{1}{n !} (\cL  \1_{[0,\infty)})^{(n)}(\lambda) A_{\lambda,n+l+1} \notag \\
&\hphantom{=}+e^{\lambda t} \!\!\!\!\! \sum\limits_{\substack{n,l\geq 0, \\ n+l<\k(\lambda)}} \frac{t^l}{l !} \frac{1}{n !} (\cL  (g_{\varepsilon}-\1_{[0,\infty)}))^{(n)}(\lambda) A_{\lambda,n+l+1}.
\label{eq:f_eps and f}
\end{align}
We continue by estimating the coefficients $\frac{1}{n !} \; (\cL  (g_{\varepsilon}-\1_{[0,\infty)}))^{(n)}(\lambda)$, $n \in \N_0$
of the Laurent series of $\cL(g_{\varepsilon}-\1_{[0,\infty)})$
around $z=\lambda$.
We have
\begin{align}
\cL(\1_{[0,\infty)}-g_\varepsilon)(z)
&= \frac1z - \frac{1-e^{-\varepsilon z}}{\varepsilon z^2}
= \frac{1}{\varepsilon z^2} \sum_{k=2}^\infty \frac{(-\varepsilon z)^k}{k!}
= \sum_{k=0}^\infty \frac{(-1)^k}{(k+2)!} \varepsilon^{k+1} z^k	\label{eq:estimate on smoothing diff Laplace}	\\
&= \sum_{k=0}^\infty \frac{(-1)^k}{(k+2)!} \varepsilon^{k+1} (z-\lambda+\lambda)^k
= \sum_{k=0}^\infty \frac{(-1)^k}{(k+2)!} \varepsilon^{k+1} \sum_{n=0}^k \binom{k}{n} \lambda^{k-n} (z-\lambda)^n	\notag	\\
&= \sum_{n=0}^\infty \bigg(\sum_{k=n}^\infty \frac{(-1)^k}{(k+2)!} \varepsilon^{k+1} \binom{k}{n} \lambda^{k-n}\bigg) (z-\lambda)^n	\notag	\\
&= \sum_{n=0}^\infty \bigg(\sum_{k=0}^\infty \frac{1}{(k+n+1)(k+n+2)} \frac{(-\varepsilon \lambda)^{k}}{k!} \bigg) \frac{(-1)^n \varepsilon^{n+1}}{n!} (z-\lambda)^n.	\notag
\end{align}
In particular, the coefficient of $(z-\lambda)^n$ of $\cL(\1_{[0,\infty)}-g_\varepsilon)(z)$ is bounded in absolute value by
\begin{equation*}
\bigg|\frac{1}{n!} \; (\cL  (g_{\varepsilon}-\1_{[0,\infty)}))^{(n)}(\lambda)\bigg|
\leq \bigg|\sum_{k=0}^\infty \frac{1}{(k+n+1)(k+n+2)} \frac{(-\varepsilon \lambda)^{k}}{k!} \frac{(-1)^n \varepsilon^{n+1}}{n!}\bigg|
\leq e^{\varepsilon |\lambda|} \frac{\varepsilon^{n+1}}{(n+2)!}.
\end{equation*}
We conclude that
\begin{align*}
\bigg|e^{\lambda t}\sum\limits_{\substack{n,l\geq 0, \\ n+l<\k(\lambda)}} \frac{t^l}{l !} \;  \frac{1}{n !} \; (\cL  (g_{\varepsilon}-\1_{[0,\infty)}))^{(n)}(\lambda) \; a^{ij}_{\lambda,n+l+1}\bigg|
&\leq e^{\Real(\lambda) t}\sum\limits_{\substack{n,l\geq 0, \\ n+l<\k(\lambda)}} \frac{t^l}{l !} \;  e^{\varepsilon |\lambda|} \frac{\varepsilon^{n+1}}{(n+2)!} \; |a^{ij}_{\lambda,n+l+1}|	\\
&= \varepsilon O(e^{\Real(\lambda) t} t^{\k(\lambda)-1})	\quad	\text{as } t \to \infty
\end{align*}
where $O(e^{\Real(\lambda) t} t^{\k(\lambda)-1})$ is an error term independent of $\varepsilon \in (0,1]$.
Summarizing,
\begin{align}
\text{Res}_{z=\lambda}(e^{t z} \cL \bU_{\!\varepsilon}(z))
&= e^{\lambda t} \sum \limits_{\substack{n,l\geq 0, \\ n+l<\k(\lambda)}} \frac{t^l}{l !}   \frac{1}{n !}  (\cL \1_{[0,\infty)})^{(n)}(\lambda)  A_{\lambda,n+l+1} \nonumber \\
&\hphantom{=}+e^{\lambda  t} \sum\limits_{\substack{n,l\geq 0, \\ n+l<\k(\lambda)}} \frac{t^l}{l !}    \frac{1}{n !} \cL (g_{\varepsilon}-\1_{[0,\infty)})^{(n)}(\lambda) A_{\lambda,n+l+1} \nonumber  \\
&= e^{\lambda t} \sum_{l=0}^{\k(\lambda)-1} t^l C_{\lambda, l}+\varepsilon O(e^{\Real(\lambda) t} t^{\k(\lambda)-1}),
\label{eq:expansion of residues}
\end{align}
where $C_{\lambda, l}$ is a $p \times p$ matrix (depending only on the parameters $\lambda$ and $l$)
and $O(e^{\Real(\lambda) t} t^{\k(\lambda)-1})$
is, in slight abuse of notation,
shorthand for a $p \times p$ matrix every entry of which is $O(e^{\Real(\lambda) t} t^{\k(\lambda)-1})$ as $t \to \infty$.

Next, let us focus on the integral on the right-hand side of \eqref{eq:LT with inversion and residues}.
Write $C<\infty$ for the supremum in \eqref{eq:sup||(Ip-Lmu(theta+ieta)^-1||<infty}.
Then
\begin{align}
\bigg\| \int_{\vartheta-\imag\infty}^{\vartheta+\imag\infty} e^{t z}  \cL \bU_{\!\varepsilon}(z) \, \dz \biggl\|
&\leq \int_{\vartheta-\imag\infty}^{\vartheta+\imag\infty} |e^{t z}|  \lVert  \cL \bU_{\!\varepsilon}(z)  \rVert  \, \dz \nonumber  \\
&\leq \int_{\vartheta-\imag\infty}^{\vartheta+\imag\infty} |e^{t z}|  | \tfrac{1-e^{-\varepsilon z}}{\varepsilon z^2} |  \lVert (\Ip-\cL\bmu(z))^{-1} \rVert \,  \dz \nonumber \\
&\leq C \, e^{t\vartheta}  \int_{\vartheta-\imag\infty}^{\vartheta+\imag\infty} \tfrac{|1-e^{-\varepsilon  z}|}{|\varepsilon z^2|} \, \dz \nonumber \\
&\leq C \,  e^{t\vartheta}  \int_{\vartheta-\imag\infty}^{\vartheta+\imag\infty} \tfrac{|\varepsilon  z| \wedge 2 }{|\varepsilon z^2|} \, \dz \nonumber \\
&\leq 2C \, e^{t\vartheta}  \int_{\varepsilon\vartheta-\imag\infty}^{\varepsilon\vartheta+\imag\infty} |z|^{-1} \wedge |z|^{-2} \, \dz \nonumber \\
&\leq 2C \, e^{t\vartheta} \int_{-\infty}^{\infty} x^{-1}\wedge x^{-2} \wedge (\varepsilon  |\vartheta|)^{-1} \, \dx \nonumber \\
&\leq 2C_\vartheta \, e^{t\vartheta} (|\log(\varepsilon)|+1).
\label{eq:2nd integral bound}
\end{align}
The constant $C_\vartheta$ in the last step does not depend on either $t$ or $\varepsilon$.
If we replace $t$ by $t+\varepsilon$ in \eqref{eq:expansion of residues},
we get
\begin{align}
\Big\|&\text{Res}_{z=\lambda}(e^{(t+\varepsilon) z}  \cL \bU_{\!\varepsilon}(z))-\text{Res}_{z=\lambda}(e^{t z}  \cL \bU_{\!\varepsilon}(z))\Big\|  \nonumber \\
&= e^{\Real(\lambda)t}\bigg\| e^{\lambda \varepsilon} \sum_{l=0}^{\k(\lambda)-1} (t+\varepsilon)^l \; C_{\lambda, l}
- \sum_{l=0}^{\k(\lambda)-1} t^l  C_{\lambda, l}\bigg\| +\varepsilon  O(e^{\Real(\lambda) t} t^{\k(\lambda)-1}) \nonumber \\
&= \varepsilon O(e^{\Real(\lambda) t} t^{\k(\lambda)-1})
\label{eq:Res at t+eps-Res at t}
\end{align}
\commentblock{
and, similarly, using \eqref{eq:expansion of residues, lambda=0}
\begin{align}
\big\|&\text{Res}_{z=0}(e^{(t+\varepsilon) z} \, \cL \bU_{\!\varepsilon}(z))-\text{Res}_{z=0}(e^{t z} \, \cL \bU_{\!\varepsilon}(z))\big\|
= \varepsilon \; O(t^{\k(0)})
\label{eq:Res at t+eps-Res at t, lambda=0}
\end{align}
for $\lambda=0$.}
Combining \eqref{eq:2nd integral bound}, \eqref{eq:Res at t+eps-Res at t}, we infer
\begin{align*}
0 \leq \|\bU_{\!\varepsilon}(t+\varepsilon)-\bU_{\!\varepsilon}(t)\|
\leq \sum_{\lambda \in \Lambda_\vartheta}O(\varepsilon e^{\Real(\lambda) t} t^{\k(\lambda)} + |\log(\varepsilon)| \; e^{t\vartheta} ).
\end{align*}
From this, \eqref{eq:U_eps(t) leq U(t) leq U_eps(t+eps)}, \eqref{eq:LT with inversion and residues} and with $\varepsilon \defeq e^{-\sigma t}$, we finally deduce
\begin{align*}
\bU(t)=\sum_{\lambda\in\Lambda_{\vartheta}}e^{\lambda  t} \sum_{k=0}^{\k(\lambda)-1} C_{\lambda,k} t^k + O(t e^{\vartheta t}).
\end{align*}

\noindent
It remains to remove the additional assumption that $\vartheta > 0$.
We do this using an exponential tilting.
Suppose that $\vartheta \leq 0$.
Then choose an auxiliary parameter $\gamma < \vartheta$ and define $\bmu^\gamma(\dx) \defeq e^{-\gamma x} \bmu(\dx)$
as well as $\bU^\gamma(\dx) \defeq e^{-\gamma x} \bU(\dx)$.
Then
\begin{equation*}
\bU^\gamma(\dx) = \sum_{n=0}^\infty (\bmu^\gamma)^{*n}(\dx).
\end{equation*}
Further, $\cL\bmu^\gamma(z) = \cL\bmu(z+\gamma)$ for all $z \in \C$ with $z+\gamma \in \Dom(\cL\bmu)$.
In particular, $\Dom(\cL\bmu^\gamma) = \Dom(\cL\bmu)-\gamma$
and, similarly, $\Lambda^\gamma \defeq \{\lambda \in \Dom(\cL\bmu^\gamma): \det(\Ip-\cL\bmu^\gamma(\lambda))=0\} = \Lambda-\gamma$ as well as $\Lambda^\gamma_{\vartheta-\gamma} = \Lambda_\vartheta-\gamma$.
Since $\vartheta - \gamma > 0$, we have an asymptotic expansion for $\bU^\gamma$, namely,
\begin{align*}
\bU^\gamma(t)
=\sum_{\lambda\in\Lambda_{\vartheta}} e^{(\lambda-\gamma) t} \sum_{k=0}^{\k(\lambda)-1} C_{\lambda,k}^\gamma t^k + O(t e^{(\vartheta-\gamma) t})
\end{align*}
where $\k(\lambda) = \mathtt{k}^\gamma(\lambda-\gamma)$ in obvious notation should be noticed
and $C_{\lambda,k}^\gamma$ is the coefficient matrix for the term $e^{\lambda t} t^k$ in the asymptotic expansion of $\bU^\gamma$.
Hence, using integration by parts,
\begin{align*}
\bU(t)
&= \int_{[0,t]} e^{\gamma x} \, \bU^\gamma(\dx)
= \int_{[0,t]} \int_0^x \gamma e^{\gamma y} \, \dy \, \bU^\gamma(\dx)
= \int_0^t  (\bU^\gamma(t)-\bU^\gamma(y)) \gamma e^{\gamma y} \, \dy	\\
&= e^{\gamma t} \bU^\gamma(t) - \int_0^t  \bU^\gamma(y) \gamma e^{\gamma y} \, \dy	\\
&=\sum_{\lambda\in\Lambda_{\vartheta}} e^{\lambda t} \sum_{k=0}^{\k(\lambda)-1} C_{\lambda,k}^\gamma t^k + O(t e^{\vartheta t})
- \gamma \sum_{\lambda\in\Lambda_{\vartheta}} \sum_{k=0}^{\k(\lambda)-1} C_{\lambda,k}^\gamma \int_0^t e^{\lambda y} y^k \, \dy
- \int_0^t O(y e^{\vartheta y}) \, \dy	\\
&= \sum_{\lambda\in\Lambda_{\vartheta}} \sum_{k=0}^{\k(\lambda)-1} C_{\lambda,k}^\gamma
\bigg(e^{\lambda t}  t^k
- \gamma \int_0^t e^{\lambda y} y^k \, \dy\bigg)
+ O(t e^{\vartheta t}).
\end{align*}
Now use that \eqref{eq:int_0^tx^ke^lambdaxdx} in the case $\lambda \not = 0$
and the elementary fact that $\int_0^t y^k \, \dy = \frac1{k+1} t^{k+1}$ if $\lambda = 0$
and rearrange terms to infer an expansion of the form \eqref{eq:asymptotic expansion U(t) non-lattice}.
\end{proof}

\begin{proof}[Proof of Theorem \ref{Thm:Markov renewal theorem F(t) non-lattice}]
First suppose that $\theta >0$.
By Theorem \ref{Thm:Markov renewal theorem U(t) non-lattice}
there exist matrices
$C_{\lambda,k} \in \R^{p \times p}$, $k=0,\ldots,\k(\lambda)-1$, $\lambda \in \Lambda_\vartheta$
such that
\begin{equation*}
\bU(t)
= \sum_{\lambda\in\Lambda_{\theta}}e^{\lambda  t} \sum_{k=0}^{\k(\lambda)-1} \!\!\! C_{\lambda, k} t^k +  O(e^{\theta t})
\quad	\text{as } t \to \infty.
\end{equation*}
Now let $f: \R \to \R^p$ be a nonnegative, right-continuous function, vanishing on the negative halfline
and with non-decreasing components on the positive halfline satisfying
\begin{equation*}
\int_{0}^\infty \!\!\! f(x)e^{-\theta x} \, \dx < \infty 
\end{equation*}
where this inequality has to be understood componentwise.
Then we may view $f$ as the measure-generating function of a column vector $\nu$ of locally finite measures on $[0,\infty)$
by letting
\begin{equation*}
f(y) \eqdef \nu([0,y]) = \int \1_{[x,\infty)}(y) \, \nu(\dx),	\quad	y \geq 0.
\end{equation*}
For the associated solution $F = \bU*f$ of the Markov renewal equation \eqref{eq:Markov renewal},
we infer
\begin{align}
F(t)
&= \bU*f(t) = \int\limits_{[0,\infty)} \bU(\dy) \, \nu([0,t-y])=\int \limits_{[0,t]} \bU(t-x) \, \nu(\dx) 	\notag	\\
&=
\int \limits_{[0,t]} \bigg(\!\sum_{\lambda\in\Lambda_{\theta}}e^{\lambda  (t-x)} \sum_{k=0}^{\k(\lambda)-1} \!\!\! C_{\lambda, k} (t-x)^k
+ O(e^{\theta (t-x)}) \! \bigg) \, \nu(\dx)	\notag	\\
&=
\sum_{\lambda\in\Lambda_{\theta}} \sum_{k=0}^{\k(\lambda)-1} \!\!\! C_{\lambda, k} e^{\lambda t} \int \limits_{[0,t]} (t-x)^k e^{-\lambda x} \, \nu(\dx)
+ O( e^{\theta t})	\label{eq:expansion 1st}
\end{align}
where the remainder term is $O( e^{\theta t})$ as $t \to \infty$ since 
\begin{align*}
\int \limits_{[0,t]} e^{\theta (t-x)} \, \nu(\dx) 
\leq e^{\theta t} \int _{[0,\infty)} e^{-\theta x} \, \nu(\dx)
= e^{\theta t} \int f(x) \theta e^{-\theta x} \, \dx < \infty.
\end{align*}
Further, for any $k \in \N$,
\begin{align*}
\int \limits_{[0,t]} (t-x)^k e^{-\lambda x} \, \nu(\dx)
&= \int \limits_{[0,t]} \int_x^t \big(k(t-y)^{k-1} + \lambda (t-y)^k\big) e^{-\lambda y} \, \dy \, \nu(\dx)	\\
&= k \int_{0}^t f(y) (t-y)^{k-1} e^{-\lambda y} \, \dy
+ \lambda \int_0^t f(y) (t-y)^k e^{-\lambda y} \, \dy
\end{align*}
and similarly
\begin{align*}
\int \limits_{[0,t]}  e^{-\lambda x} \, \nu(\dx)
&= f(t)e^{-\lambda t} + \lambda \int_{0}^tf(y) e^{-\lambda y}\, \dy.
\end{align*}
Hence, rearranging the terms in \eqref{eq:expansion 1st} and defining the $B_{\lambda,k}$ appropriately
gives \eqref{eq:expansion F(t) non-lattice}.
Finally, for $k \in \N_0$, we expand
\begin{align*}
\int \limits_0^t f(x) (t-x)^k e^{-\lambda x} \, \dx
&= \sum_{j=0}^k (-1)^{k-j} \binom{k}{j} t^j \int \limits_0^t f(x) x^{k-j} e^{-\lambda x} \, \dx
\end{align*}
and notice that
\begin{align*}
\int \limits_0^t x^{k-j} e^{-\lambda x} \, \dx
&= \int \limits_0^\infty f(x) x^{k-j} e^{-\lambda x} \, \dx - \int \limits_t^\infty f(x) x^{k-j} e^{-\lambda x} \, \dx
\end{align*}
where
\begin{align*}
\bigg|e^{\lambda t} \int \limits_t^\infty f(x) x^{k-j} e^{-\lambda x} \, \dx\bigg|
&\leq \int \limits_t^\infty f(x) x^{k-j} e^{\Real(\lambda) (t-x)} \, \dx \leq  e^{(\theta+\varepsilon) t} \int \limits_t^\infty f(x) e^{-\theta x} \, \dx
\end{align*}
for all sufficiently large $t$ and $0 < \varepsilon \leq \Real(\lambda)-\theta$.
On the other hand, 
\begin{align*}
	f(t)\le e^{\theta (t+1)}\int_{t}^{t+1}f(x)e^{-\theta x} \, \dx=O(e^{\theta t}).
\end{align*}
Using this in \eqref{eq:expansion F(t) non-lattice}
gives \eqref{eq:expansion F(t) non-lattice 2}.
The proof can be extended to the case $\theta < 0$ using a similar strategy as in proof of the Theorem \ref{Thm:Markov renewal theorem U(t) non-lattice}. We refrain from providing further details.

Finally, if $f: [0,\infty) \to \R^p$ is a function satisfying \eqref{eq:int e^-theta x f(x) < infty},
we can write it as a difference $f=f_+-f_-$ of non-decreasing, right-continuous functions.
This is the Jordan decomposition of $f$. The identity $\mathrm{V}\! f (x) = f_+(x)+f_-(x)$
imply that both, $f_+$ and $f_-$ satisfy \eqref{eq:int e^-theta x f(x) < infty}.
The previous proof can therefore be applied to both, $f_+$ and $f_-$,
and then extends by linearity to $f$.
\end{proof}

\subsection{Asymptotic expansions in the lattice case}	\label{subsec:asymptotic expansions lattice}

\begin{proof}[Proof of Lemma \ref{Lem: Multiplicities of lattice/non-lattice}]
Let $\lambda \in \Lambda$ and $A_{\lambda, k(\lambda)}$ be as in \eqref{eq:expansion of I_p-Lmu inverse at lambda}.
By definition, $A_{\lambda, k(\lambda)}$ is finite and has a non-zero entry.
We show that $(\Ip-\cG\bmu(z))^{-1}$ also has a pole with multiplicity $\k(\lambda)$ at $e^{-\lambda}$.
Indeed,
\begin{align}
\lim_{z\to e^{-\lambda}} (\Ip-\cG\bmu(z))^{-1}(z-e^{-\lambda})^{\k(\lambda)} &= \lim_{e^{-z}\to e^{-\lambda}} (\Ip-\cG\bmu(e^{-z}))^{-1}(e^{-z}-e^{-\lambda})^{\k(\lambda)} \nonumber \\
&=  \lim_{z \to \lambda} (\Ip-\cL\bmu(z))^{-1}(z-\lambda)^{\k(\lambda)} \frac{(e^{-z}-e^{-\lambda})^{\k(\lambda)}}{(z-\lambda)^{\k(\lambda)}}
\nonumber  \\
&= A_{\lambda, \k(\lambda)} (-e^{-\lambda})^{\k(\lambda)} \nonumber.
\end{align}
\end{proof}

\begin{proof}[Proof of Theorem \ref{Thm:Markov renewal theorem F(n) lattice}]
For $r>0$, let $B_r\defeq \{ |z|<r\}$ be the open disc with radius $r$ in $\C$
centered at the origin and $\partial B_r\defeq \{ |z|=r\}$ its boundary.
Now fix $r<e^{-\alpha}$.
Since $\cG\bmu$ has holomorphic entries on $B_{e^{-\vartheta}}$
and $\cG f$ has holomorphic entries on $B_{e^{-\theta}} \subseteq B_{e^{-\vartheta}}$
by \eqref{eq:sum e^-theta n f(n) < infty},
we infer from Cauchy's integral formula and the convolution theorem for generating functions that
\begin{align}
\bmu^{*l}*f(n)=\frac{1}{2\pi\imag} \int_{\partial B_r} \frac{\cG(\bmu^{*l}*f)(z)}{z^{n+1}} \, \dz
=\frac{1}{2\pi\imag} \int_{\partial B_r} \frac{(\cG\bmu)^l(z)\cG f(z)}{z^{n+1}} \, \dz.
\label{eq: Cauchy integral formula}
\end{align}
In particular,
\begin{align}
F(n)&=\bU * f(n)=\sum_{l=0}^{\infty} \bmu^{*l} * f(n)=\sum_{l=0}^{\infty}\frac{1}{2\pi\imag} \int_{\partial B_r} \frac{(\cG\bmu)^l(z) \cG f(z)}{z^{n+1}} \, \dz \notag \\
&=\frac{1}{2\pi\imag} \int_{\partial B_r} \frac{(\Ip-\cG\bmu(z))^{-1} \cG f(z)}{z^{n+1}} \, \dz
\label{eq: F(n)}
\end{align}
where the last step follows from Lemma \ref{Lem:rho_A<1}
since $\rho_{\cG\bmu(z)}<1$ for $|z|=r<e^{-\alpha}$. For $\lambda\in\Lambda_\vartheta$,
recall that $b^{ij}_{\lambda,n}$ is the coefficient of $(z-e^{-\lambda})^{-n}$ in the Laurent series of the $(i,j)$-th entry
of $(\Ip-\cG\bmu(z))^{-1}$ at $z=e^{-\lambda}$.
In particular, $b^{ij}_{\lambda,n}=0$ for $n>\k(\lambda)$, where $\k(\lambda)$ is the multiplicity of the singularity $\lambda$
in $(\Ip-\cG\bmu(z))^{-1}$.
Further, $B_{\lambda,k}= (b^{ij}_{\lambda,k})_{i,j\in\p}$.
We write $\cG f(z)=\sum_{l=0}^{\infty}f_l(e^{-\lambda}) (z-e^{-\lambda})^l$
with coefficients $f_l(e^{-\lambda})=\frac{1}{l!}(\cG f)^{(l)}(e^{-\lambda})$ for $l\in\N_0$.
Define the vectors $G(z)$ and $H(z)$ by 
\begin{align*}
G(z)&\defeq \sum_{\lambda\in\Lambda_\theta} \sum_{k=1}^{\k(\lambda)} B_{\lambda,k} \sum_{l=0}^{k-1} f_{l}(e^{-\lambda}) (z-e^{-\lambda})^{-k+l}, \\
H(z)&\defeq (\Ip-\cG(z))^{-1} \cG f(z) -G(z).
\end{align*}
Here $G$, is meromorphic in $\C$ with poles in $\Lambda_\vartheta$
and $H$ is holomorphic in a neighborhood of $B_{e^{-\vartheta}}$.
Further, $(\Ip-\cG(z))^{-1}\cG f(z)=G(z)+H(z)$, $z \in B_{e^{-\vartheta}} \setminus \Lambda_\vartheta$.
Consequently, with \eqref{eq: F(n)}, we infer
\begin{align}
F(n)= \frac{1}{2\pi\imag} \int_{\partial B_r} \frac{G(z)+H(z)}{z^{n+1}} \, \dz.
\label{eq: F(n) explicit}
\end{align}
By the residue theorem, for any $d\in\N_0$,
\begin{align*}
 \frac{1}{2\pi\imag} \int_{\partial B_r}  \frac{(z-e^{-\lambda})^{-d}}{z^{n+1}} \, \dz
 &= (-1)^d e^{\lambda (d + n)} \binom{n+d-1}{d-1}.
\end{align*}
Consequently,
\begin{align*}
 \frac{1}{2\pi\imag} \int_{\partial B_r} \frac{G(z)}{z^{n+1}}\dz
 &=  \sum_{\lambda\in\Lambda_\theta} \sum_{k=1}^{\k(\lambda)} B_{\lambda,k} \sum_{l=0}^{k-1} f_{l}(e^{-\lambda})  (-e^\lambda)^{-k+l} e^{\lambda n} \binom{n+k-l-1}{k-l-1} \\
&=\sum_{\lambda\in\Lambda_\theta}e^{\lambda n} \sum_{k=0}^{\k(\lambda)-1}\binom{n+k}{k} \sum_{l=0}^{\k(\lambda)-1}  B_{\lambda,k+l+1}f_{l}(e^{-\lambda})  (-e^\lambda)^{k+1}	\\
 &=\sum_{\lambda\in\Lambda_\theta}e^{\lambda n} \sum_{k=0}^{\k(\lambda)-1}b_{\lambda,k,f}n^k.
\end{align*}
Here, the vectors $b_{\lambda,k,f}$ depend only on the parameters $\lambda$, $k$ and $f$.
Now, we consider the disc $B_{e^{-\vartheta}}$ and observe that the matrix $H(z)$ has holomorphic entries.
Thus, with $\theta_n \defeq \theta + \frac1n$, we infer
\begin{align}
\bigg\| \int_{\partial B_{e^{-r}}} \frac{H(z)}{z^{n+1}} \, \dz \biggl\|
= \bigg\| \int_{\partial B_{e^{-\theta_n}}} \frac{H(z)}{z^{n+1}} \, \dz \biggl\|
\leq
\sup_{z \in B_{e^{-\theta}}} \|H(z)\| \frac1{(e^{-\theta_n})^n} = O(e^{\theta n})	\quad	\text{as } n \to \infty
\end{align}
and with \eqref{eq: F(n) explicit},
\begin{align}
F(n)=\sum_{\lambda\in\Lambda_\theta}e^{\lambda n} \sum_{k=0}^{\k(\lambda)-1} n^k b_{\lambda,k,f} + O(e^{\theta n}).
\end{align}
\end{proof}

\begin{proof}[Proof of Corollary \ref{Cor:Markov renewal theorem lattice}]
Suppose that (A\ref{ass:first moment}) and (A\ref{ass:subcritical instant offspring}) hold.
First consider the case $\vartheta > 0$
and fix $j \in \p$. Define $F_j(n) \defeq \bU(\{n\}) \e_j$ where $\e_j$ is the $j$th unit vector.
From $\bU(\{n\}) = \Ip + \bmu * \bU(\{n\})$ we infer
\begin{equation*}
F_j(n) = \e_j + \bmu*F_j(n),	\quad	n \in \N_0.
\end{equation*}
Since $\vartheta > 0$, we have $\sum_{n=0}^\infty \e_j e^{-\vartheta n} < \infty$.
We may thus apply Theorem \ref{Thm:Markov renewal theorem F(n) lattice} to infer
\begin{equation*}	
F_j(n) = \sum_{\lambda\in\Lambda_\theta} e^{\lambda n} \sum_{k=0}^{\k(\lambda)-1} n^k b_{\lambda,k,j}+O(e^{\vartheta n})
\quad	\text{as } n \to \infty.
\end{equation*}
Notice that we can apply the theorem in a way that the error term is of the order $O(e^{\vartheta n})$ since
$\vartheta$ is an inner point of $\Dom(\cL\bmu)$ by (A\ref{ass:first moment}).
Write $C_{\lambda,k}$ for the matrix with columns $b_{\lambda,k,1},\ldots,b_{\lambda,k,p}$. Then
\begin{equation*}	
\bU(\{n\}) = \sum_{\lambda\in\Lambda_\theta} e^{\lambda n} \sum_{k=0}^{\k(\lambda)-1} n^k C_{\lambda,k}+O(e^{\vartheta n})
\quad	\text{as } n \to \infty
\end{equation*}
as claimed.
\end{proof}

\section{Application to general branching processes}	\label{sec:applications}

The Markov renewal equation appears in the study of general branching processes.
We continue with a brief introduction to this class of processes
and refer to \cite{Jagers:1975} for a more detailed introduction.

\subsection{The multi-type general branching process}	\label{subsec:General branching process}

In this section, we formally introduce the multi-type general (Crump-Mode-Jagers) branching process.
The process starts with a single individual, the ancestor, born at time $0$.
The ancestor gets the label $\varnothing$, the empty tuple, and, like every individual in the process,
has one of finitely many types $\tau(\varnothing) \in \p$ where $p \in \N$ is the number of different types.
We write $\bxi=(\xi^{i,j})_{i,j = 1,\ldots,p}$ for the $p \times p$ matrix of point processes $\xi^{i,j}$, $i,j\in\p$
that determine the ancestor's reproduction. Here, a point process is a random locally finite point measure, see \cite[Chapter 3]{Resnick:2008}.
If $\tau(\varnothing) = i$, i.e., if the type of the ancestor is $i$,
then its children of type $j$ are born at the times of the point process $\xi^{i,j}$,
$j=1,\ldots,p$.
We enumerate the children of the ancestor with the labels $1,2,\ldots,N$ where
$N$ is the total number of offspring of the ancestor
and can formally be expressed as
$N = \sum_{j=1}^p \xi^{ij}([0,\infty)) \in \N_0 \cup \{\infty\}$ on $\{\tau(\varnothing)=i\}$.
We write $X_k$ and $\tau(k)$ for the birth time and type of individual $k$, respectively.
We can make it so that the enumeration is such that $0 \leq X_1 \leq X_2 \leq \ldots$.
For completeness, we write $X_k = \infty$ if individual $k$ is never born.

The multi-type general branching process is now obtained in that the children of the ancestor
produce children according to independent copies of the point process $\bxi$
who in turn have children who reproduce according to further independent copies of $\bxi$, and so on.

We can then assign a label from the infinite Ulam-Harris tree $\I \defeq \bigcup_{n \in \N_0} \N^n$ to each individual
$u=(u_1,\ldots,u_n) = u_1\ldots u_n \in \I$
according to its ancestral line, namely,
\begin{align*}
\varnothing \rightarrow u_1 \rightarrow u_1u_2 \rightarrow ... \rightarrow u_1...u_n=u,
\end{align*} 
where $u_1$ is the $u_1$-th child of the ancestor, $u_1u_2$ is the $u_2$-th child of $u_1$, and so on.
We write $|u|=n$ if $u \in \N^n$.

Now let $\bxi_u=(\xi_u^{i,j})_{i,j = 1,\ldots,p}$, $u \in \I \setminus \{\varnothing\}$
be a family of i.\,i.\,d.\ copies of $\bxi \eqdef \bxi_\varnothing$.
We suppose that $(\Omega,\A,\Prob)$ is a probability space on which all $\bxi_u$, $u \in \I$ are defined and i.\,i.\,d.\ and
independent of $\tau(\varnothing)$, the random variable that gives the ancestor's type.
We write $\Prob^i$ if the ancestor's type is $i \in \p$.
We denote the associated expected value operators by $\E$ and $\E^i$, respectively.

Each individual $u \in \familytree$ that is ever born has a clearly defined type $\tau(u) \in \p$
and a time of birth $S(u)$, namely, $S(\varnothing)=0$ and, recursively,
\begin{align*}
S(uk) = S(u) + X_{u,k},	\quad	u \in \I,\, k \in \N.
\end{align*}
For completeness, we write $S(u)=\infty$ if individual $u$ is never born.

Now suppose that for each $u \in \I$ there exist a random vector $\bzeta_u=(\zeta_u^1,\ldots,\zeta_u^p)^\transp$
taking values in $[0,\infty]^p$ and a product-measurable, separable random characteristic
$\bphi_u = (\varphi_u^1,\ldots,\varphi_u^p)^\transp: \Omega \times \R \to \R^{p}$
(taking values in the space of $p$-dimensional real column vectors)
such that the $(\bxi_u,\bzeta_u)$, $u \in \I$ are i.i.d.\ and the $\bphi_u$, $u \in \I$ are identically distributed.
Further, we assume throughout that for every $n \in \N$, the $\bphi_u$, $u \in \N^n$ are independent and independent of
the $(\bxi_u,\bzeta_u,\bphi_u)$, $|u|<n$. These assumptions are satisfied, for instance, when the $(\bxi_u,\bzeta_u,\bphi_u)$, $u \in \I$
are i.\,i.\,d. However, our assumption is weaker and allows $\bphi_u$ to be a function of the $(\bxi_{uv},\bzeta_{uv})$, $v \in \I$.
The variable $\zeta_u^i$ is viewed as the lifetime of the potential individual $u$ given that its type is $i$.
The function $\varphi_u^i(t)$ is some kind of score assigned to the individual $u$ at the age $t$ given that its type is $i$.
Hence, for $u \in \familytree$, we define $\zeta_u \defeq \bzeta_u^{\tau(u)}$ as the lifetime of individual $u$
and $\varphi_u(t) \defeq \varphi_u^{\tau(u)}(t) = \e_{\tau(u)}^\transp \bphi_u(t)$ as the score assigned to individual $u$ at the age $t$.
For $t \geq 0$ and $j \in \p$, define
\begin{align*}
\cZ_t^{j} \defeq \sum_{u \in \familytree} \1_{\{j\}}(\tau(u)) \1_{[0,\zeta_u)}(t-S(u)),
\end{align*}
i.e., $\cZ_t^{j}$ is the number of individuals of type $j$ alive at time $t$. Now write
\begin{align*}
\cZ_t
&\defeq
(\cZ_t^1, \ldots, \cZ_t^p)
\end{align*}
for the associated row vector.
Then $(\cZ_t)_{t \geq 0}$ is the multi-type general branching process.
We are interested in the asymptotic behavior of $\E^i[\cZ_t]$ as $t \to \infty$
for $i=1,\ldots,p$.
More generally, we are interested in the asymptotic behavior of $\E^i[\cZ_t^{\bphi}]$ as $t \to \infty$
for the multi-type general branching process counted with characteristic $\bphi$ defined via
\begin{align*}
\cZ_t^{\bphi} \defeq \sum_{u \in \familytree} \varphi_u(t-S(u)),	\quad	t \in \R.
\end{align*}
Notice that introducing a general score $\bphi$ indeed generalizes the model since we may write
$\cZ_t^j = \cZ_t^{\bphi}$ for $\bphi_u$ with
\begin{align*}
\varphi_u^i(t)
&=	\begin{cases}
	0				&	\text{for } i \not = j,	\\
	\1_{[0,\zeta^j_u)}	&	\text{for } i = j.
	\end{cases}
\end{align*}
Similarly, the number $N_t^j$ of all individuals of type $j$ born up to and including time $t$,
we can write in the form $\cZ_t^{\bphi}$ by choosing $\bphi_u$ as
\begin{align*}
\varphi_u^i(t)
&=	\begin{cases}
	0			&	\text{for } i \not = j,	\\
	\1_{[0,\infty)}	&	\text{for } i = j.
	\end{cases}
\end{align*}

\subsection{Recursive decomposition of general branching processes}	\label{subsec:general branching process}

The expectation of a multi-type Crump-Mode-Jagers counted with random characteristic satisfies a
Markov renewal equation as already observed in \cite{Jagers:1975} in the single-type case.

Suppose that $(\cZ_t^{\bphi})_{t \geq 0}$ is a multi-type Crump-Mode-Jagers process counted with random characteristic
$\bphi$. Define $F_i(t) \defeq \E^i[\cZ_t^{\bphi}(t)]$ for $i=1,\ldots,p$ and $F(t) \defeq (F_1(t),\ldots,F_p(t))^\transp$, $t \in \R$.
If $\bphi$ satisfies appropriate integrability conditions to be specified later, then $F:\R \to \R^p$ vanishes on the negative halfline
and satisfies a Markov renewal equation,
namely, with $S_k(u) = S(ku)-X_k$ on $\{X_k<\infty\}$, by conditioning with respect to the first generation,
\begin{align*}
F_i(t)
&= \E^i[\cZ_t^{\bphi}(t)]
= \E^i\bigg[\sum_{u \in \familytree} \varphi_u(t-S(u)) \bigg]
= \E^i\bigg[\varphi^i(t) + \sum_{k=1}^N \sum_{u \in \familytree} \varphi_{ku}(t-X_k-S_k(u)) \bigg]	\\
&= \E^i[\varphi^i(t)] + \sum_{j=1}^p \int F_j(t-x) \, \mu^{i,j}(\dx)
= f_i(t) + (\bmu * F(t))_i
\end{align*}
where $(\bmu * F(t))_i$ is the $i$-th entry of the vector $\bmu * F(t)$ and $f_i(t) \defeq \E^i[\varphi^i(t)]$.
In other words, $F = f + \bmu * F$, that is, $F$ satisfies \eqref{eq:Markov renewal}
provided that $F(t)$ and $f(t)$ exist and are finite.

Below we give the applications of our main results to the general branching process,
limiting ourselves to the non-lattice case. The lattice case is analogous.

\begin{theorem}	\label{Thm:Markov renewal theorem E[N_t] non-lattice}
Consider a multi-type general branching process with
matrix of intensity measures $\bmu=(\mu^{i,j})_{i,j\in\p}$
satisfying (A\ref{ass:first moment}), (A\ref{ass:subcritical instant offspring}) and
\eqref{eq:sup||(Ip-Lmu(theta+ieta)^-1||<infty}.
Further, suppose that $\Lambda_\vartheta$ is finite.
Write $M(t) = (m_{ij}(t))_{i,j\in \p} = (\E^i[N_t^j])_{i,j \in \p}$ where $\E^i[N_t ^j]$ is the expected number of individuals of type $j$
born up to and including time $t$ given that the ancestor's type is $i$.
Then there exist matrices $C_{\lambda,k} \in \R^{p \times p}$, $k=0,\ldots,\k(\lambda)-1$, $\lambda \in \Lambda_\vartheta$,
$C_{0,\k(0)}$ such that
\begin{equation}	\label{eq:asymptotic expansion E[N_t] non-lattice}
M(t)
= \sum_{\lambda\in\Lambda_{\vartheta}}e^{\lambda  t} \sum_{k=0}^{\k(\lambda)-1} \!\!\! t^k C_{\lambda, k}
+ t^{\k(0)} C_{0,\k(0)} + O(t e^{\vartheta t})
\quad	\text{as } t \to \infty,
\end{equation}
where the error bound $O(t e^{\vartheta t})$ as $t \to \infty$ applies entrywise.
\end{theorem}
\begin{proof}
The result follows immediately from Theorem \ref{Thm:Markov renewal theorem U(t) non-lattice}
since $M(t) = \bU(t) = \sum_{n=0}^\infty \bmu^{*n}(t)$ for all $t \geq 0$. 
\end{proof}

\begin{theorem}	\label{Thm:Markov renewal theorem E[Z_t^phi] non-lattice}
Consider a multi-type Crump-Mode-Jagers branching process $(\cZ_t^{\bphi})_{t \geq 0}$
counted with random characteristic $\bphi$.
Suppose that the matrix $\bmu=(\mu^{i,j})_{i,j\in\p}$ of intensity measures satisfies
(A\ref{ass:first moment}), (A\ref{ass:subcritical instant offspring}) and \eqref{eq:sup||(Ip-Lmu(theta+ieta)^-1||<infty}.
Further, suppose that $\Lambda_\vartheta$ is finite.
Assume that $f(t) \defeq (\E^1[\varphi^1](t),\ldots,\E^p[\varphi^p](t))^\transp$ vanishes on the negative halfline, is finite for every $t \geq 0$
and right-continuous with existing left limits as a function of $t$.
Further, suppose that
\begin{align} \nonumber
\int_0^\infty \!\!\! e^{-\theta x} \, \mathrm{V}\! f(x)  \, \dx < \infty
\end{align}
for some $\theta > \vartheta$ with $\Lambda_\vartheta=\Lambda_\theta$.
Then there exist matrices $B_{\lambda,k} \in \R^{p \times p}$, $k=0,\ldots,\k(\lambda)-1$, $\lambda \in \Lambda_\vartheta$,
and vectors $b_{\lambda,k} \in \R^p$, $k=0,\ldots,\k(\lambda)-1$, $\lambda \in \Lambda_\vartheta$ such that, for $F(t) \defeq (\E^1[\cZ_t^{\bphi}],\ldots,\E^p[\cZ_t^{\bphi}])^\transp$,
it holds that
\begin{align}
F(t) &=
\sum_{\lambda\in\Lambda_{\theta}} \!e^{\lambda t} \!\! \sum_{k=0}^{\k(\lambda)-1} \!\!\! B_{\lambda, k} \! \int \limits_0^t \! f(x) (t-x)^k e^{-\lambda x} \, \dx + O(e^{\theta t})	\nonumber	\\
&= \sum_{\lambda\in\Lambda_{\theta}} e^{\lambda t} \!\! \sum_{k=0}^{\k(\lambda)-1} \!\!\! t^k b_{\lambda, k}
+ O(e^{(\theta+\varepsilon) t})	\nonumber
\end{align}
as $t \to \infty$, where the error bounds $O(e^{\theta t})$ and $O(e^{(\theta+\varepsilon) t})$ as $t \to \infty$ are to be read component by component
and where $\varepsilon>0$ can be chosen arbitrarily small, in particular, $\theta+\varepsilon < \Real(\lambda)$ for every $\lambda \in \Lambda_\theta$.
\end{theorem}

\subsection{Examples}	\label{subsec:Examples}

We start with an example demonstrating that if $\cL\bmu(\theta)$ is not primitive,
then the Malthusian parameter may have multiplicity strictly larger than $1$.

\begin{example}	\label{Exa:m'(alpha)=0}
Let $p=2$ and suppose that $\xi^{1,1}$ and $\xi^{2,2}$ are homogeneous
Poisson point processes on $[0,\infty)$ with rate $a>0$
and that $\xi^{1,2}$ is a homogeneous Poisson point process with rate $b > 0$.
Finally, let $\xi^{2,1}=0$. Then
\begin{align*}
\cL\bmu(z) =
\begin{pmatrix}
a	& b 	\\
0	& a
\end{pmatrix}
\frac1z,	\quad	\Real(z)>0.
\end{align*}
Then
\begin{align*}
\det(\mathbf{I}_2-\cL\bmu(z))
= \det
\begin{pmatrix}
1-\frac{a}{z}	& -\frac{b}{z} 	\\
0	& 1-\frac{a}{z}
\end{pmatrix}
= \big(1-\tfrac{a}{z}\big)^2.
\end{align*}
Hence, $\alpha = a$, $\Lambda = \{\alpha\}$ and $\alpha$ has multiplicity $2$.
\end{example}

The next example is to point out that the order of the pole at a zero $\lambda$
of the inverse $(\Ip-\cL\bmu(z))^{-1}$ can be strictly smaller than the order of the zero of the determinant $\det(\Ip-\cL\bmu(z))$.

\begin{example}	\label{Exa:multiplicity of determinant and Laurent series}
Let $m:[0,\infty) \to (0,\infty)$ be the Laplace transform of a locally finite measure on $[0,\infty)$
with $1 < m(\vartheta) < \infty$ for some $\vartheta \in \R$ and $\lim_{\theta \to \infty} m(\theta) < 1$.
Consider  
\begin{align*}
\cL\bmu(z) =
\begin{pmatrix}
m(z)	& 0 	\\
0	& m(z)
\end{pmatrix},	\quad	\Real(z) \geq \vartheta.
\end{align*}
Then
\begin{align*}
\det(\mathbf{I}_2-\cL\bmu(z)) =
\det
\begin{pmatrix}
1-m(z)	& 0 	\\
0		& 1-m(z)
\end{pmatrix}
= (1-m(z))^2.
\end{align*}
Any root of the determinant is, therefore, a root of even multiplicity.
In particular, since $m'(\alpha) \in (-\infty,0)$,
the Malthusian parameter $\alpha$ is a zero of the determinant with multiplicity $2$.
On the other hand,
\begin{align*}
(\mathbf{I}_2-\cL\bmu(z))^{-1} =
\begin{pmatrix}
\tfrac1{1-m(z)}	& 0 	\\
0			& \tfrac1{1-m(z)}
\end{pmatrix}
=
\begin{pmatrix}
\tfrac1{-m'(\alpha)}	& 0 	\\
0				& \tfrac1{-m'(\alpha)}
\end{pmatrix}
(z-\alpha)^{-1} + \begin{pmatrix}
h(z)	& 0 	\\
0				& h(z)
\end{pmatrix}
\end{align*}
where $h(z)$ is holomorphic at $z=\alpha$.
In other words, $(\mathbf{I}_2-\cL\bmu(z))^{-1}$ has a pole of order $1$ at $z=\alpha$.
\end{example}

\begin{example}
Consider a $2$-type process, in which particles of type $1$ give birth to particles of type $1$
according to a Poisson point process $\xi_{11}$ with intensity $\alpha>0$ and particles of type $2$ give birth
to particles of type $2$ also according to a Poisson point process $\xi_{22}$ with intensity $\alpha>0$
where $\xi_{11}$ and $\xi_{22}$ are independent. Further, let $\xi_{12}=\delta_0$ and $\xi_{21}=0$,
i.e., at the time of birth, every type-$1$-particle immediately produces a type-$2$-particle.
Then
\begin{align*}
\cL\bmu(z) =
\begin{pmatrix}
\tfrac{\alpha}z	& 1 	\\
0		& \tfrac{\alpha}z
\end{pmatrix},	\quad	\Real(z) > 0.
\end{align*}
Further, $\det(\mathbf{I}_2-\cL\bmu(z)) = (1-\frac{\alpha}z)^2$
and $\alpha$ is a zero of order two of the determinant $\det(\mathbf{I}_2-\cL\bmu(z))$.
Moreover,
\begin{align*}
(\mathbf{I}_2-\cL\bmu(z))^{-1}
&=
\frac{1}{(1-\frac{\alpha}{z})^2}
\begin{pmatrix}
1-\tfrac{\alpha}z	& 1 	\\
0		& 1-\tfrac{\alpha}z
\end{pmatrix}	\\
&= 
\begin{pmatrix}
0	& \alpha^2 	\\
0	& 0
\end{pmatrix} \frac{1}{(z-\alpha)^2}
+ 
\begin{pmatrix}
\alpha	& 2\alpha	\\
0		& \alpha
\end{pmatrix} \frac{1}{z-\alpha}
+
\begin{pmatrix}
1	& 1 	\\
0	& 1
\end{pmatrix}
,	\quad	\Real(z) > 0,\ z \not = \alpha.
\end{align*}
This means that $(\mathbf{I}_2-\cL\bmu(z))^{-1}$ has a pole of order $2$ at $z=\alpha$ and $\Lambda=\{\alpha\}$.
Consequently, we infer 
\begin{align*}
	\E[N_t]&=e^{\alpha t}\begin{pmatrix}
					1	&   1+t \alpha	\\
					0	& 1
					\end{pmatrix}+O(te^{\vartheta t})
\end{align*}
for any $\vartheta > 0$ with Theorem \ref{Thm:Markov renewal theorem E[N_t] non-lattice}
and the coefficients from Remark \ref{Rem:C_lambda,k}.
Indeed, the assumptions of the aforementioned theorem clearly hold and by the  remark we have
\begin{equation*}
C_{\alpha,0}
= \frac{1}{\alpha} A_{\alpha,1}-\frac{1}{\alpha^2} A_{\alpha,2}=
\begin{pmatrix}
1	&   1	\\
0	& 1
\end{pmatrix},
\end{equation*}
and
\begin{equation*}
C_{\alpha,1}
=  \frac{1}{\alpha} A_{\alpha,2}=
\begin{pmatrix}
0	&   \alpha	\\
0	& 0
\end{pmatrix}.
\end{equation*}
\smallskip

\noindent
There is an equivalent way to describe the above process, namely,
type-$1$ and type-$2$ particles reproduce according to independent homogeneous
Poisson processes $\xi_{11}$ and $\xi_{22}$ with intensity $\alpha>0$,
but instead of saying that every type-$1$ particle produces a type-$2$ particle directly at birth,
we say that whenever a type-$1$ particle is created (only type-$1$ particles produces type-$1$ particles),
a type-$2$ particle is also produced. This is equivalent to setting $\xi_{12} \defeq \xi_{11}$ (instead of $\delta_0$)
if for every type $1$-particle in the initial configuration, an additional type-$2$ particle is included.
The resulting process has the exact same particle configuration at any time whenever instead of starting with a particle of type-$1$ we start with two particles: one of type $1$ and the other of type $2$ and hence $(1,0)\E[N_t]=(1,1)\E[\widetilde N_t]$. The
 associated Laplace transform
$\cL\widetilde\bmu$ is
\begin{align*}
\cL\widetilde\bmu(z) =
\begin{pmatrix}
\tfrac{\alpha}z	& \tfrac{\alpha}{z} 	\\
0		& \tfrac{\alpha}z
\end{pmatrix},	\quad	\Real(z) > 0.
\end{align*}
Again, $\det(\mathbf{I}_2-\cL\widetilde\bmu(z)) = (1-\frac{\alpha}z)^2$, i.e., $\Lambda = \{\alpha\}$
and $\alpha$ is a zero of order two of the determinant $\det(\mathbf{I}_2-\cL\widetilde\bmu(z))$.
The Laurent expansion of $(\mathbf{I}_2-\cL\widetilde\bmu(z))^{-1}$ is now given by
\begin{align*}
(\mathbf{I}_2-\cL\widetilde\bmu(z))^{-1}
&=
\frac{1}{(1-\frac{\alpha}{z})^2}
\begin{pmatrix}
1-\tfrac{\alpha}z	& \tfrac\alpha{z} 	\\
0		& 1-\tfrac{\alpha}z
\end{pmatrix}	\\
&= 
\begin{pmatrix}
0	& \alpha^2 	\\
0	& 0
\end{pmatrix} \frac{1}{(z-\alpha)^2}
+ 
\begin{pmatrix}
\alpha	& \alpha	\\
0	& \alpha
\end{pmatrix} \frac{1}{z-\alpha}
+
\begin{pmatrix}
	1	& 0	\\
	0	& 1
\end{pmatrix} 
,
\quad	\Real(z) > 0,\ z \not = \alpha.
\end{align*}
The same argument as above leads to, in obvious notation,
\begin{equation*}
	\widetilde C_{\alpha,0}
	= \frac{1}{\alpha} \widetilde A_{\alpha,1}-\frac{1}{\alpha^2} \widetilde A_{\alpha,2}=
	\begin{pmatrix}
		 1 	&   0	\\
		0	&  1 
	\end{pmatrix},\quad
	\widetilde C_{\alpha,1}
	=  \frac{1}{\alpha} \widetilde A_{\lambda,2}=
	\begin{pmatrix}
		0	&   \alpha	\\
		0	& 0
	\end{pmatrix}
\end{equation*}
and as a result
\begin{align*}
	\E[\widetilde N_t]&=e^{\alpha t}\begin{pmatrix}
		 1 	&   t \alpha	\\
		0	&  1 
	\end{pmatrix}+O(te^{\vartheta t})
\end{align*}
\end{example}

\section*{Acknowledgement}
Matthias Meiners and Ivana Tomic were supported by DFG grant ME3625/4-1.
Part of this work was conducted during Matthias Meiners and Ivana Tomic's visit to the University of Wroc\l aw, for which they express their gratitude for the warm hospitality.

\begin{appendix}

\section{Auxiliary results concerning matrices}

In this section we gather some basic facts about matrices that we need somewhere in the paper.
These are not new. For the reader's convenience,
we provide either a source or (a sketch of) a proof.


Recall that a nonnegative $d \times d$ matrix $A$ is called \emph{primitive}
if $A^n$ has positive entries only for some nonnegative integer $n$.

Recall that
\begin{align}	
\varrho: \Dom(\cL\bmu) \to [0,\infty),	\quad	z \mapsto \varrho(z) \defeq \rho_{\cL \bmu(z)}.
\end{align}

We continue with a brief summary of properties of $\varrho$.

\begin{proposition}	\label{Prop:properties of varrho}
Suppose that (A\ref{ass:first moment}) holds.
\begin{enumerate}[(a)]
	\item
		The function $\varrho|_{\Dom(\cL\bmu) \cap \R}: \Dom(\cL\bmu) \cap \R \to [0,\infty)$
		is logarithmic convex (i.e., $\log \varrho$ is convex) and continuous.
	\item
		$\lim_{\theta \to \infty} \varrho(\theta)=\rho_{\bmu(0)}$.

	\item
		The function $\varrho|_{\Dom(\cL\bmu) \cap \R}: \Dom(\cL\bmu) \cap \R \to [0,\infty)$ is strictly decreasing
		on the set
		$\{\theta \in \Dom(\cL\bmu) \cap \R: \varrho(\theta) > \rho_{\bmu(0)}\}$.
	\item
		For every $z \in \Dom(\cL\bmu)$, we have $\varrho(z) \leq \varrho(\Real(z))$.
\end{enumerate}
\end{proposition}

Before we prove this, we first provide some basic facts regarding the spectral radius.

\begin{lemma}	\label{Lem:largest eigenvalue cts and mon}
Let $(A_s)_{s \in I}$ be a family of nonnegative $d \times d$ matrices with decreasing and
continuous entries $a_s^{ij}: I \to [0,\infty)$,
$i,j=1,\ldots,d$ where $I \subseteq \R$ is a non-empty interval. Write $\rho_s \defeq \rho(A_s)$ for the modulus
of the largest eigenvalue of $A_s$.
\begin{enumerate}[(a)]
	\item
		The function $s \mapsto \rho_s$ is decreasing and continuous on $I$.
	\item
		If the $A_s$, $s \in I$ are primitive, then for every $s \in I$, there exists
		a unique normalized eigenvector $v_s$ associated with the Perron-Frobenius eigenvalue $\rho_s$ of $A_s$
		with positive entries only.
		The function $s \mapsto v_s$ is continuous on $I$.
	\item
		If the $A_s$, $s \in I$ are primitive and $\sum_{i,j=1}^d a_s^{ij}$ is strictly decreasing,
		then $s \mapsto \rho_s$ is strictly decreasing.
\end{enumerate}
\end{lemma}
\begin{proof}
In this proof, we will repeatedly use the fact that if $p_\varepsilon(z)$ is a polynomial
of degree $\leq d-1$ for $\varepsilon \geq 0$ and the coefficients
of $p_\varepsilon(z)$ converge to those of $p_0(z)$ as $\varepsilon \downarrow 0$,
then the zeros of $z^d + p_\varepsilon(z)$ converge to those of $z^d + p_0(z)$ as $\varepsilon \downarrow 0$,
see e.g.\ Theorem (1,4) of \cite{Morris:1949} (a corollary of Rouch\'e's theorem).	\smallskip

\noindent
(a) and (c):
We start by showing that $s \mapsto \rho_s$ is continuous.
To this end, notice that $\rho_s$ is the largest root (in absolute value) of the characteristic polynomial
$\chi_s(z) = \det(z \Id - A_s)$.
The characteristic polynomial is of the form
$\chi_s(z) = z^d + p_{s}(z)$ where $p_{s}(z)$ is a polynomial of degree $\leq d-1$
with coefficients that are continuous in $s$. We conclude that $s \mapsto \rho_s$ is continuous.

In the next step, we prove the claimed monotonicity statements.
To this end, first suppose that the $A_s$, $s \in I$ are primitive.
Then $\rho_s$ is the Perron-Frobenius eigenvalue of $A_s$ and $s \mapsto \rho_s$
is nonincreasing by the Perron-Frobenius theorem \cite[Theorem 1.1(e)]{Seneta:1981}.
The latter also gives that $\rho_s$ is strictly decreasing
if the sum of all entries of $A_s$ is strictly decreasing.

The next item of business is to prove monotonicity without the additional assumption that the $A_s$, $s \in I$ are primitive.
Indeed, we may write $A_s^\varepsilon$ for the matrix $A_s$
after adding $\varepsilon$ to every non-zero entry of the matrix $A_s$
where $\varepsilon \geq 0$ is a parameter. Then $A_s = A_s^0$.
Write $\rho_s^\varepsilon \defeq \rho(A_s^\varepsilon)$.
Since $A_s^\varepsilon$ is primitive, $\rho_s^\varepsilon$ is (strictly) decreasing in $s$
by what we have already shown.
On the other hand, again by Theorem (1,4) of \cite{Morris:1949},
we infer $\rho_s^\varepsilon \to \rho_s$ as $\varepsilon \downarrow 0$.
Consequently, $s \mapsto \rho_s$ is decreasing as the pointwise limit of a sequence of decreasing functions.
The proof of (a) and (c) is therefore complete.	\smallskip

\noindent
(b) Suppose  the $A_s$, $s \in I$ are primitive.
The existence of a unique normalized eigenvector $v_s$ with positive entries only
associated with the Perron-Frobenius eigenvalue $\rho_s$ of $A_s$
then directly follows from the Perron-Frobenius theorem \cite[Theorem 1.1]{Seneta:1981}.
It remains to show the continuity of $s \mapsto v_s$ on $I$.
First, again by the Perron-Frobenius theorem, there exists a unique
with $w_s = (w_s^1,\ldots,w_s^d)^\transp \in \R^d$ with $A_s w_s = \rho_s w_s$ and $w_s^1=1$.
Then $v_s = w_s/|w_s|$. This means that
\begin{align*}
\begin{pmatrix}
a_s^{21} \\ \vdots \\ a_s^{d1}
\end{pmatrix}
+
\begin{pmatrix}
a_s^{22}	& \ldots	& a_s^{2d} \\
\vdots	& 		& \vdots \\
a_s^{d2}	& \ldots	& a_s^{dd}
\end{pmatrix}
\begin{pmatrix}
w_s^{2} \\ \vdots \\ w_s^{d}
\end{pmatrix}
= \rho_s
\begin{pmatrix}
w_s^{2} \\ \vdots \\ w_s^{d}
\end{pmatrix}
\end{align*}
We write this identity $a_s' + A_s' w_s'= \rho_s w_s'$ and $w_s' \in \R^{d-1}$.
This may be rewritten in the form $(\rho_s \mathbf{I}_{d-1} - A_s') w_s' = a_s'$.
The matrix $(\rho_s \mathbf{I}_{d-1} - A_s')$ is invertible as $\rho_s$ is not an eigenvalue of $A_s'$.
Indeed, $\rho(A_s') < \rho(A_s) = \rho_s$.
To see this in detail, consider
$C_0$ to be the $d \times d$ matrix with $A_s'$ in the bottom right corner and $0$'s in the first row and column.
Write $C_t$ for the matrix that has $A_s'$ in the bottom right corner and $a_{ij} \wedge t$
as $(i,j)$th entry for $i=1$ or $j=1$. Then $C_t \uparrow A_s$ as $t \uparrow \max\{a_{i,j}: i=1 \text{ or } j=1\}$
and $C_t \downarrow C_0$ as $t \downarrow 0$.
By part (a), $\rho(C_t)$ is increasing and by part (c), $\rho(C_t)$ is even strictly increasing
where the sum of the entries of $C_t$ is strictly increasing, which is the case in a neighborhood of the origin.
The claim follows.
We now infer
\begin{align*}
w_s' = (\rho_s \mathbf{I}_{d-1} - A_s')^{-1} a_s'.
\end{align*}
The right-hand side is a continuous function of $s \in I$
since $\rho_s$, $A_s'$ and $a_s'$ are continuous.
Therefore $w_s'$ and thus also $w_s$ and $v_s = w_s/|w_s|$ is a continuous function of $s$.
\end{proof}

\begin{proof}[Proof of Proposition \ref{Prop:properties of varrho}]
(a) The logarithmic convexity follows from the main theorem of \cite{Kingman:1961}.
Logarithmic convexity implies convexity and hence continuity on the interior of the domain, but we claim continuity on the entire domain,
which may contain the left boundary point.
However, the continuity on the entire domain follows from Lemma \ref{Lem:largest eigenvalue cts and mon}.	\smallskip

\noindent
(b) Notice that $\cL\mu^{i,j}(\theta) \to \mu^{i,j}(0)$ as $\theta \to \infty$ for all $i,j \in \p$.
In other words, $\lim_{\theta \to \infty} \cL\bmu(\theta) = \bmu(0)$.
We may thus extend the function $\varrho$ to $\infty$ by setting $\varrho(\infty) \defeq \rho_{\bmu(0)}$.
Lemma \ref{Lem:largest eigenvalue cts and mon} still applies and yields that the extended $\varrho$ is continuous
on $(\Dom(\mu)\cap\R) \cup \{\infty\}$, hence
\begin{align}	\label{eq:rho_mu(theta)->rho_mu(0)}
\varrho(\theta) = \rho_{\cL\bmu(\theta)} \to \rho_{\bmu(0)}	\quad	\text{as }	\theta \to \infty.
\end{align}

\noindent
(c) That $\varrho$ is decreasing again follows from Lemma \ref{Lem:largest eigenvalue cts and mon}.
Consequently, if $\varrho$ is not strictly decreasing on an interval, it is constant there.
By convexity and (b), $\varrho$ must be strictly decreasing on $\{\theta \in \Dom(\cL\bmu) \cap \R: \varrho(\theta) > \rho_{\bmu(0)}\}$.	\smallskip

\noindent
(d) follows from Lemma \ref{Lem:rho is monotone} below.
\end{proof}

\begin{lemma}	\label{Lem:rho_A<1}
Let $A$ be a nonnegative $d \times d$ matrix with $\rho(A) < 1$.
Then the series $\sum_{n=0}^\infty A^n$ converges and equals $(\Id-A)^{-1}$.
\end{lemma}
\begin{proof}
%
%
This is  \cite[Section 11.1, Theorem 1]{Lancaster+Tismenetsky:1985}.
\end{proof}

\begin{lemma}	\label{Lem:rho is monotone} 
Let $A =(a_{i,j})_{i,j =1,\ldots,d} \in \C^{d \times d}$ and $B = (b_{i,j})_{i,j=1,\ldots,d}$ be a nonnegative $d \times d$ matrix
with $|a_{i,j}| \leq b_{i,j}$ for all $i,j=1,\ldots,d$. Then $\rho(A) \leq \rho(B)$.
\end{lemma}
\begin{proof}
This is \cite[Section 15.2, Theorem 1]{Lancaster+Tismenetsky:1985}.
\end{proof}

\section{Auxiliary results concerning Laplace transforms and integrals}

The following lemma is probably well known but we could not provide a suitable source in the literature.
For the reader's convenience, we include a short proof. 

\begin{lemma}	\label{Lem:growth of fcts with finite LT}
Let $f:[0,\infty) \to [0,\infty)$ be a nondecreasing function with Laplace transform $\cL f$ finite in $\theta > 0$,
i.e.,
\begin{align*}
\cL f (\theta) = \int_0^\infty e^{-\theta x} f(x) \, \dx < \infty.
\end{align*}
Then $f(t) = o(e^{\theta t})$ as $t \to \infty$.
\end{lemma}
\begin{proof}
For any $t > 0$, we have
\begin{align*}
f(t)
&= \frac1\theta \int \limits_0^\infty e^{-\theta x} f(t) \, \dx
= \frac1\theta e^{\theta t} \int \limits_0^t e^{-\theta (x+t)} f(t) \, \dx + \frac1\theta \int \limits_t^\infty e^{-\theta x} f(t) \, \dx	\\
&\leq \frac1\theta e^{\theta t} \int \limits_t^{2t} e^{-\theta x} f(x) \, \dx + \frac1\theta \int \limits_t^\infty e^{-\theta x} f(x) \, \dx
= o(e^{\theta t})	\quad	\text{as } t \to \infty.
\end{align*}
\end{proof}

\begin{lemma}	
For $\lambda \in \C \setminus \{0\}$, $k \in \N_0$ and $t \geq 0$, we have
\begin{equation}	\label{eq:int_0^tx^ke^lambdaxdx}
\int_0^t x^k e^{\lambda x} \, \dx = \frac{(-1)^{k+1} k!}{\lambda^{k+1}}\bigg(1-e^{\lambda t} \sum_{j=0}^k (-1)^j \frac{\lambda^j t^j}{j!}\bigg).
\end{equation}
\end{lemma}
\begin{proof}
First consider $\lambda = -\theta$ for some $\theta > 0$ and write $S_{k+1}$ for a random variable with distribution
$\Gamma(k+1,\theta)$, which is the $(k+1)$-fold convolution power of $\Exp(\theta)$.
Since the increments in a homogeneous Poisson process $(N_t)_{t \geq 0}$
with intensity $\theta>0$ are independent exponentials,
we have $\Prob(N_t \geq k+1) = \Prob(S_{k+1} \leq t)$. Hence,
\begin{align*}
\int_0^t x^{k} e^{-\theta x} \, \dx &= \frac{k!}{\theta^{k+1}} \int_0^t \frac{\theta^{k+1} x^{k}}{k!} e^{-\theta x} \, \dx
= \frac{k!}{\theta^{k+1}} \Prob(S_{k+1} \leq t) = \frac{k!}{\theta^{k+1}} (1-\Prob(N_t \leq k))	\\
&= \frac{k!}{\theta^{k+1}}\bigg(1-e^{-\theta t} \sum_{j=0}^k \frac{(\theta t)^j}{j!}\bigg).
\end{align*}
This is \eqref{eq:int_0^tx^ke^lambdaxdx} in the special case $\lambda = -\theta < 0$.
Since both sides of \eqref{eq:int_0^tx^ke^lambdaxdx} are holomorphic functions in $\lambda \in \C \setminus \{0\}$,
the assertion follows from the uniqueness theorem for holomorphic functions.
\end{proof}

\end{appendix}


%
\bibliographystyle{abbrv}

\bibliography{Branching_processes}	 


\end{document}